\documentclass[11pt,twoside, leqno]{article}

\usepackage{amssymb}
\usepackage{amsmath}
\usepackage{mathrsfs}
\usepackage{amsthm}
\usepackage{amsfonts}
\usepackage{txfonts}
\usepackage{enumerate}
\usepackage{hyperref}
\usepackage{latexsym,amssymb}
\usepackage{color}
\usepackage{indentfirst}
\usepackage{graphicx}
\usepackage{pgf,tikz}
\usepackage{mathrsfs}
\usepackage{float}
\usetikzlibrary{arrows}
\usepackage{graphicx}    
\usepackage{caption}     
\usepackage{subcaption}  
\usepackage{float}       

\definecolor{ttttff}{rgb}{0.2,0.2,1.}
\definecolor{ttffcc}{rgb}{0.2,1.,0.8}
\definecolor{qqqqff}{rgb}{0.,0.,1.}
\definecolor{zzttqq}{rgb}{0.6,0.2,0.}
\definecolor{qqqqff}{rgb}{0.,0.,1.}

\allowdisplaybreaks

\def\rn{{\mathbb{R}^n}}

\newtheorem{theorem}{Theorem}[section]
\newtheorem{lemma}[theorem]{Lemma}
\newtheorem{corollary}[theorem]{Corollary}
\newtheorem{proposition}[theorem]{Proposition}

\theoremstyle{definition}
\newtheorem{remark}[theorem]{Remark}
\newtheorem{definition}[theorem]{Definition}
\renewcommand{\appendix}
{\par
   \setcounter{section}{0}%
   \setcounter{subsection}{0}%
   \setcounter{subsubsection}{0}%
   \gdef\thesection{\@Alph\c@section}%
   \gdef\thesubsection{\@Alph\c@section.\@arabic\c@subsection}%
   \gdef\theHsection{\@Alph\c@section.}%
   \gdef\theHsubsection{\@Alph\c@section.\@arabic\c@subsection}%
   \csname appendixmore\endcsname
 }

\numberwithin{equation}{section}

\definecolor{qqqqff}{rgb}{0,0,1}
\definecolor{ffqqqq}{rgb}{1,0,0}

\begin{document}
\arraycolsep=1pt

\title{\bf Perturbed Dyadic Cubes and Quantitative Estimates for Schr\"odinger Operators  
 with Potentials in $RH^{n/2}$ \footnotetext{\hspace{-0.35cm}
{\it 2020 Mathematics Subject Classification}. {Primary: 35J10; Secondary: 
47A55; 47A10; 47G40.}
\endgraf{\it Key words and phrases.} Schr\"{o}dinger operator, reverse H\"older class, dyadic cube, 
sparse domination, Riesz potential.
\endgraf This project is supported by the National Natural Science Foundation of China (Grant No. 12071431).
\endgraf $^*$\,Corresponding author/submitted version.
}}

\author{Jun Cao*, Cheng Chen, Chaohong Deng and Yulian Wu}

\date{ }

\maketitle

\vspace{-0.8cm}

\begin{center}
\begin{minipage}{13cm}\small
{\noindent{\bf Abstract.}
Let $L:=-\Delta+V$ be a Schr\"odinger operator on the Euclidean space $\mathbb{R}^n$ with potential 
$V$ in the reverse H\"older class $RH^{n/2}$ satisfying some mild assumptions that $V$ neither 
decays too rapidly nor oscillates violently at infinity. In this paper, the authors construct a new system of dyadic cubes 
$\mathcal{D}^V$ that reflects the intrinsic geometry perturbed by $V$. Then using the 
quantitative geometric information of $\mathcal{D}^V$,
the authors characterize the $L^p$ operator norm of the Riesz potential $L^{-\alpha/2}$ 
for all $\alpha\in (0,2]$ and $p\in (1,\infty)$. As applications, 
some quantitative spectral estimates for $L$ are given.}
\end{minipage}
\end{center}

\vspace{0.2cm}

\tableofcontents

\section{Introduction}\label{s1}
Let $\Delta$ be the Laplace operator and $V$ a function on $\mathbb{R}^n$. $-\Delta+|x|^{-2}$ 
The Schr\"odinger operator $-\Delta+V$ originates from the Hamiltonian formulation of quantum wave 
propagation \cite{s1926}. Its profound mathematical properties provide a rigorous and systematic foundation 
for quantum mechanics (\cite{n1930, w1931}). A cornerstone of modern Schr\"odinger operator theory is T. Kato’s seminal result on the essential self-adjointness of $N$-body quantum Hamiltonians \cite{k1951}. 
Mathematically, it asserts that $-\Delta+V$ is essentially self-adjoint in $L^2 (\mathbb{R}^3)$ for any potential $V\in L^2 (\mathbb{R}^3)+L^\infty (\mathbb{R}^3)$.

The key technique employed in \cite{k1951} is the perturbation theory, which regards the potential $V$ as a 
perturbation of the well-studied Laplace operator $-\Delta$. Accordingly, a fundamental question in the perturbation model is: 
how properties of the potential $V$ influence the behavior of the Schr\"odinger 
operator $-\Delta+V$ (\cite{cfks1986, hs1996, k1995}). 
If the influence of  $V$ is small in the sense that $V$ is 
{\it relatively $-\Delta$-bounded}: there exist constants $a\in (0,1)$ and $b\in [0,\infty)$ such that, 
for any $u\in C_c^\infty(\mathbb{R}^n)$,
\begin{align}\label{rb}
\|V u\|_{L^2(\mathbb{R}^n)} \leq a \|\Delta u\|_{L^2(\mathbb{R}^n)} + b \|u\|_{L^2(\mathbb{R}^n)},
\end{align}
then the Kato-Rellich theorem 
(see, e.g., \cite[Theorem 13.5]{hs1996}) implies that  the operator $-\Delta+V$ is self-adjoint and shares the same domain of $-\Delta$ in $L^2(\mathbb{R}^n)$. Note that inequality \eqref{rb} is a kind of trace estimate which absorbs the perturbation 
of $V$ into the effect of $\Delta$, ensuring that $-\Delta+V$ retains many 
desirable properties of $-\Delta$.  The following Weyl's theorem (see, e.g., \cite[Theorem 13.9]{hs1996}) is another example, which shows that both $-\Delta+V$ and $-\Delta$ share the same essential spectrum. 

\medskip

\noindent{\bf Theorem A} (\cite{hs1996}). 
{\it Let $V\geq 0$ be relatively $-\Delta$-bounded as in \eqref{rb} and
$V(x) \rightarrow 0$ as $|x|\rightarrow \infty$. Then
\begin{equation*}
\sigma(-\Delta + V)=\sigma_{\rm ess}(-\Delta + V)=[0,\infty).
\end{equation*} }

\medskip

Examples of potentials $V$ with small perturbation include 
$L^2(\mathbb{R}^3)+L^\infty(\mathbb{R}^3)$ \cite{k1951}, Stummel class $S_4(\rn)$ \cite{s1956} and
Kato class $K_2(\rn)$ \cite{k1972}. In particular, let $V(x)=|x|^a$ with  $a\in \mathbb{R}$, 
then $V\in K_2(\rn)$ if and only if $a\in (-2,0]$. For $a\notin (-2,0]$,  the perturbation of $V$ may changes the structure of 
Schr\"odinger operator $-\Delta+V$ essentially; see, for instance, 
the following theorem on the distribution of discrete spectrum
(see, e.g., \cite[Theorem 10.7]{hs1996}).

\medskip

\noindent{\bf Theorem B} (\cite{hs1996}). 
{\it Let $V\geq 0$, $V\in L_{\mathrm{loc}}^2(\rn)$, and $V(x) \rightarrow \infty$ as $|x|\rightarrow \infty$.
Then $-\Delta + V$ has purely discrete spectrum.}

\medskip

The assumptions of Theorem B are fulfilled for the potential $V(x)=|x|^a$ if and only if $a\in (0, \infty)$. In this case, $V$ belongs to the 
so-called reverse H\"older class $RH^q$ with $1 < q <\infty$.  
Recall that  a nonnegative function $V\in L_{\mathrm{loc}}^q (\rn)$ is said to belong to the {\it reverse H\"{o}lder class $RH^q$}, 
if there exists $C>1$ such that, for any ball $B\subseteq \rn$,
\begin{equation*}
\left(\frac{1}{|B|}\int_B {V(x)^q dx}\right)^\frac{1}{q} \leq 
\frac{C}{|B|}\int_B {V(x)dx}.
\end{equation*}
Although  the trace estimate \eqref{rb} generally fails for $V\in RH^q$, the study of  reverse 
H\"{o}lder class is usually related to a critical radius function introduced in \cite{s1995}. 
To be precise, let $V\in RH^{n/2}$. 
For any $x\in \rn$, the {\it critical radius function} $\rho(x,V)$ is defined by 
\begin{equation}\label{eqn-rho}
\rho(x,V):= \frac{1}{m(x,V)}:= \sup\left\{r>0: \frac{1}{r^{n-2}} \int_{B(x, r)} V(y) d y \leq 1\right\}.
\end{equation}
This critical radius $\rho$ is frequently used to compensate the perturbations that the potential $V$ brings in
many classical estimates of Laplacian. The following theorem from \cite[Theorem 0.8]{s1999}
is a typical example on the fundamental solution estimates.

\medskip

\noindent{\bf Theorem C} (\cite{s1999}). 
{\it Let $V\in RH^{n/2}$. Suppose that $\Gamma_{V}(x,y)$ is a 
fundamental solution of $-\Delta+V$ in $\rn$. Then
\begin{equation}\label{eqn-FSE}
\frac{C_1 e^{-\varepsilon_1 [1+(|x-y|/\rho(x,V))]^{k_{0}+1}}}{|x-y|^{n-2}}
\leq \Gamma_{V}(x,y) \leq
\frac{C_2 e^{-\varepsilon_2 [1+(|x-y|/\rho(x,V))]^{1/(k_{0}+1)}}}{|x-y|^{n-2}}, 
\end{equation}
where $C_1$, $C_2$, $\varepsilon_1$, $\varepsilon_2$ and $k_{0}$ are 
positive constants that are independent of 
$x$ and $y$.}

\medskip 

Note that if $V(x)\equiv 0$, then $\rho(x,V)=\infty$. In this case, \eqref{eqn-FSE} reduces to the following well-known classical estimates for 
fundamental solution $\Gamma_{0}$ of Laplacian (see, e.g., \cite{e2010})
\begin{align*}
\Gamma_{0}(x,y) \sim
\frac{1}{|x-y|^{n-2}}.
\end{align*}
Based on the  compensate estimates bringed by $\rho$ as in \eqref{eqn-FSE}, a systematical study of 
Schr\"odinger operator $-\Delta+V$ with $V\in RH^q$ is established 
(see, e.g., \cite{bhs2011, clyz2022, dz1999, glp2008, l1999, ld2010, t2015, ww2025, wy2016, yz2011}). 

In this paper, we introduce a new method to study the Schr\"odinger 
operator $-\Delta+V$ in the framework of perturbation theory. Unlike the analytic trace estimates \eqref{rb} and 
compensate estimates as in \eqref{eqn-FSE}, we focus on the geometric aspect of the perturbation theory. The idea is that 
we regard each Schr\"odinger operator as the canonical ``Laplacian" of the underlying intrinsic geometry, which 
is a perturbed version of the classical flat geometry corresponding to Laplace operator. 
Thus, once the properties of the perturbed 
geometry have been established, the analysis of the Schr\"odinger operator can be conducted in a canonical  ``Laplacian" manner. Moreover, in view of the fundamental role played by the dyadic cubes in flat geometry, 
it is natural to ask the following question.

\medskip

\noindent{\bf Question}: Can we construct a system of perturbed dyadic cubes for 
Schr\"odinger operators $-\Delta+V$?

\medskip 

In this article, we provide an affirmative answer to this question 
under the following three assumptions on $V$:
\begin{enumerate}
\item[($\mathrm{\bf A_1}$)]
there exists $r_0>0$ such that for all $x\in\mathbb{R}^n$ satisfying $|x|\geq r_0$, it holds
\begin{equation*}
\frac{1}{|x|^{n - 2}}\int_{B(x,|x|)}{V(y)\,dy}\gtrsim 1;
\end{equation*}
\item[($\mathrm{\bf A_2}$)]
there exists $C_D>0$ such that for all $r_1\geq r_0/2$ and $x,y\in\mathbb{R}^n$ satisfying 
$r_1\leq |x|\leq|y|<2r_1$, \\ it holds
\begin{equation*}
V(y)\leq C_D V(x);
\end{equation*}
\item[($\mathrm{\bf A_3}$)]
 there exists $C_{RD}>0$ such that for all $r_1\geq r_0/2$ and $x,y\in\mathbb{R}^n$ satisfying 
 $r_1\leq |x|\leq|y|< 2r_1$, it holds
\begin{equation*}
V(x)\leq C_{RD} V(y).
\end{equation*}
\end{enumerate}

\begin{remark}\label{rem-1}
Intuitively, Assumptions $\mathbf{(A_1)}$–$\mathbf{(A_3)}$ mean that the potential $V$ neither 
decays too rapidly nor oscillates violently.
In the language of critical radius function $\rho(x,V)$, this indicates that 
\begin{align*}
\rho(x,V)\lesssim |x|
\end{align*}
when $|x|\geq r_0$ and $\rho(x,V)$ is comparable in each far-field annulus.
In particular, let $V(x)=|x|^a$ with $a\in \mathbb{R}$. It is well-known that
$V \in RH^{n/2}$ if and only if $a\in (-2, \infty)$ (see \cite[p. 516]{s1995}). 
For such $a$, it is straightforward to verify that 
$V$ satisfies Assumptions $\mathbf{(A_1)}$–$\mathbf{(A_3)}$ with 
the critical radius function
\begin{equation}\label{eqn-ECR}
\rho(x,V)\sim 
\begin{cases}
1, &|x|<1, \\
|x|^{-\frac{a}{2}}, &|x|\geq1.
\end{cases}
\end{equation}
 

\end{remark}

The first main result of this article is the following construction of perturbed system of
dyadic cubes associated with potential $V$ satisfying 
Assumptions $\mathbf{(A_1)}$–$\mathbf{(A_3)}$.

\begin{theorem}\label{t1.2}
Let $V\in RH^{n/2}$ satisfy Assumptions $\mathbf{(A_1)}$–$\mathbf{(A_3)}$.
Then there exists a family $\mathcal{D}^V := \bigcup_{k\in \mathbb{Z}_+} \mathcal{D}_k^V$ 
of cubes in $\rn$ that satisfies the following properties:
\begin{enumerate}
\item[\textnormal{(i)}]
for each $k\in \mathbb{Z}_+$, the family $\mathcal{D}_k^V$ of cubes forms a partition of $\rn$. 
Moreover, the interiors of cubes in $\mathcal{D}_k^V$ are pairwise disjoint;
\item[\textnormal{(ii)}]
for any $k>l$, $Q\in \mathcal{D}_l^V$ and $P\in \mathcal{D}_k^V$, it holds
$Q\cap P=\emptyset$ or $Q\cap P =P$;
\item[\textnormal{(iii)}]
for any $k\in \mathbb{Z}_+$, any $Q\in \mathcal{D}_k^V$ and every $x\in Q$, 
we have $l(Q)\sim 2^{-k}\rho(x,V)$.
\end{enumerate}
\end{theorem}

\begin{figure}[H]
  \centering  
  \begin{subfigure}[b]{0.45\textwidth}
    \centering
    \captionsetup{labelformat=empty}  
    \includegraphics[scale=0.14]{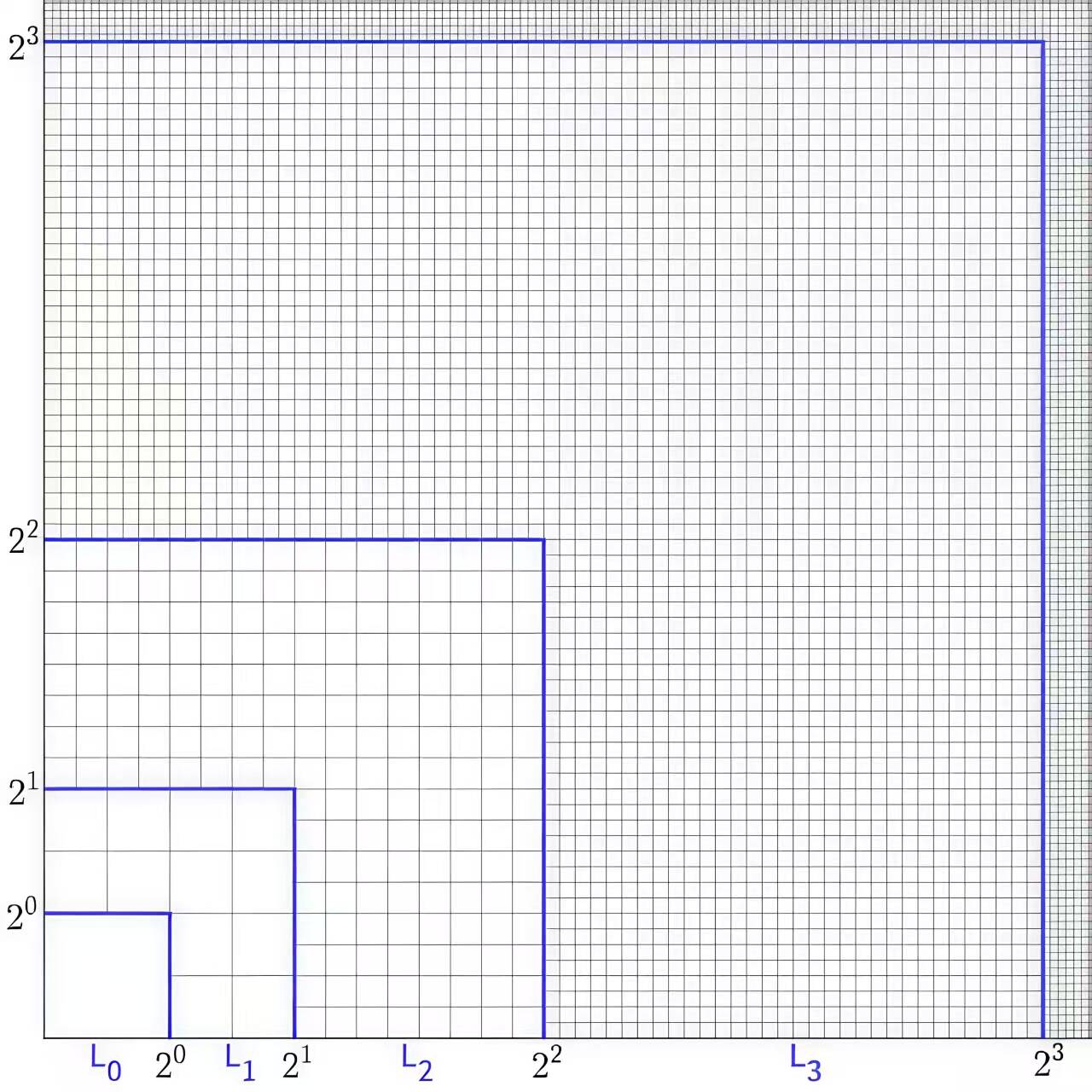}
    \caption{Figure 1: Dyadic cubes $\mathcal{D}^V$ for $V(x)=|x|^2$} 
  \end{subfigure}
  \hfill 
  \begin{subfigure}[b]{0.45\textwidth}
  \captionsetup{labelformat=empty}  
    \centering
    \includegraphics[scale=0.15]{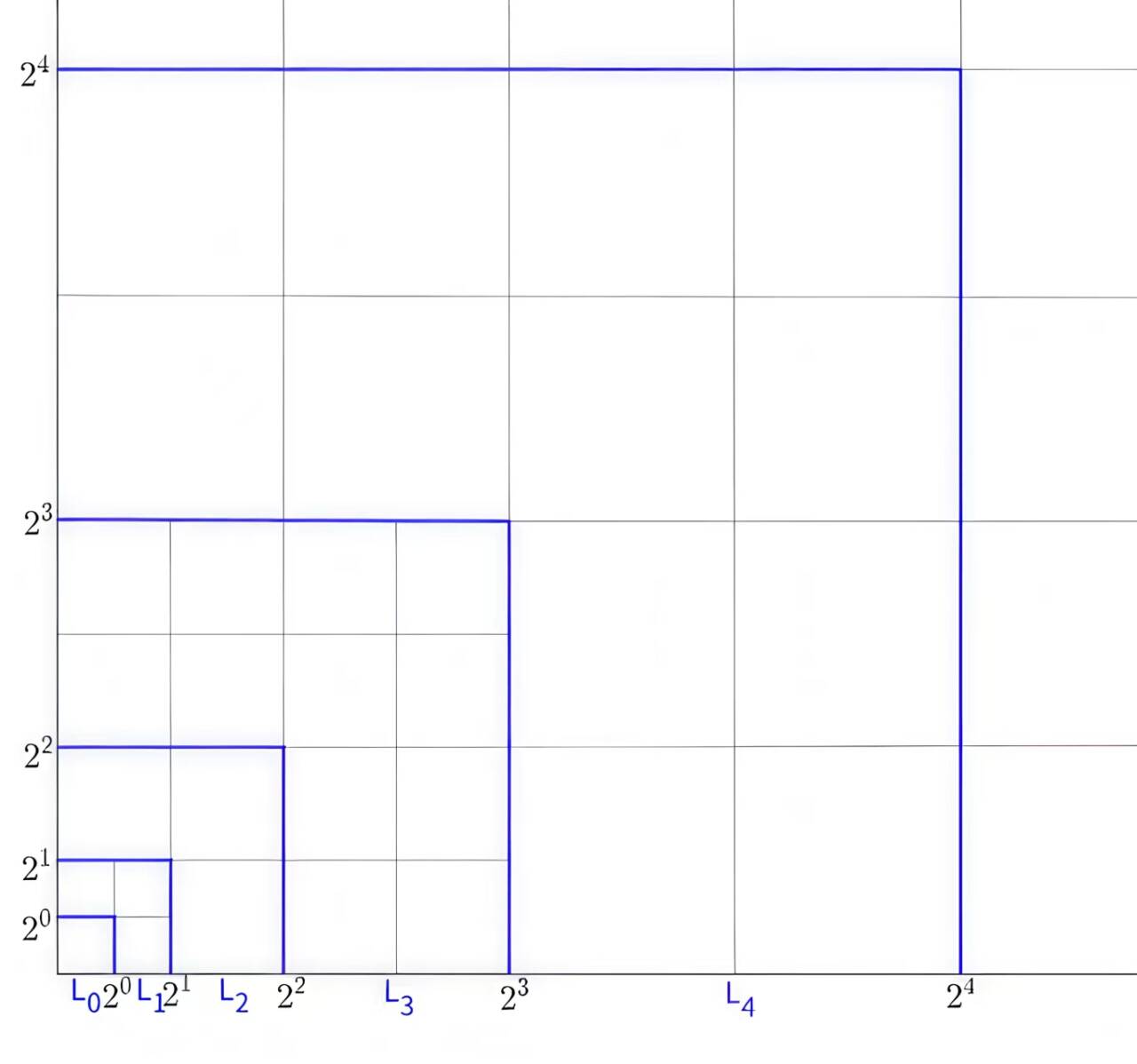}
    \caption{Figure 2: Dyadic cubes $\mathcal{D}^V$ for $V(x)=|x|^{-1}$} 
  \end{subfigure}
\end{figure}

Recall that the classical dyadic cube system $\mathcal{D}$ is defined for any $k\in \mathbb{Z}$ as
\begin{equation}\label{eqn-CDC}
\mathcal{D}:=\bigcup\limits_{k\in \mathbb{Z}} \mathcal{D}_{k}, \quad \text{where} \quad \mathcal{D}_k 
= \left\{ 2^{-k}\bigl([0,1)^n+j\bigr) : j\in \mathbb{Z}^n \right\}.
\end{equation}
All cubes of $\mathcal{D}^V$ are from  
$\mathcal{D}$, but they are arranged in a way based on a layer structure 
$\{L_l\}_{l\in\mathbb{Z}_+}$ based on the critical radius function 
$\rho(x,V)$ (see \eqref{eqn-Lay} for more details).
This perturbation give rise to an anisotropic property of $\mathcal{D}^V$: cubes of 
the same generation may of different sizes if they are in different layers. 
See Figures 1 and 2 for illustrations of the perturbed 
dyadic cubes of the $0$-th generation associated with the potential 
$V(x)=|x|^a$ for $a=2$ and $a=-1$, respectively.  
This property reflects the key influence that 
potential $V$ brings in the construction of $\mathcal{D}^V$, which cannot be 
obtained from the general construction in the setting of metric measure space \cite{hk2012}. 
We point out that similar ideas have also appeared in some related works such 
as the Whitney covering for Gaussian measure \cite{mnp2012} 
(correspond to the case $V(x)=|x|^2$) and atomic decomposition of 
Hardy space associate with Schr\"odinger operator \cite{dz1999}. 
An interesting property that perturbed 
dyadic cubes bring is that if we consider the Muckenhoupt weight class $A_p^{\mathcal{D}^V}$ 
associated $\mathcal{D}^V$ for all $p\in [1,\infty)$ in the sense that 
\begin{equation*}
\sup_{Q \in \mathcal{D}^V} 
\left( \frac{1}{|Q|} \int_{Q} \omega(x) \, dx \right)
\left( \frac{1}{|Q|} \int_{Q} \omega(x)^{-1/(p-1)} \, dx \right)^{p-1} < \infty
\end{equation*}
with the usual modification when $p=1$. 
From the properties of $\mathcal{D}^V$, the Gaussian function $e^{-|x|^2}$, 
which is not a classical Muckenhoupt weight \cite{ur2019}, 
but belong to $A_1^{\mathcal{D}^V}$ with $V(x) = |x|^2$ 
(see Lemma \ref{lem3.x7}).
 
Once the perturbed dyadic system $\mathcal{D}^V$ is constructed, we can treat
the Schr\"odinger operator $-\Delta+V$ in a ``Laplacian" way by using geometric 
properties of  $\mathcal{D}^V$. Recall that the classical dyadic system $\mathcal{D}$ 
plays a fundamental role in quantitative estimates for many classical operators
associated to $\Delta$ 
(see, e.g., \cite{cdj2025, cirxy2024, cg2024, h2012, lmpt2010, lms2014, l2013, l2016, th2024}).

The extension of $\mathcal{D}$ to  Schr\"odinger operator 
$\Delta+V$ usually need a compensate argument as in Theorem {\bf C}.
Recall that
the {\it Riesz potential $(-\Delta+V)^{-\alpha/2}$}, for any $\alpha\in(0,n)$, 
associated with $-\Delta+V$ is defined by setting for any $p\in (1,\infty)$,
$f\in L^p(\rn)$,
\begin{equation}\label{eqn-fractionalp}
(-\Delta+V)^{-\frac{\alpha}{2}}f
:=\int_0^\infty e^{-t(-\Delta+V)}f\,\frac{dt}{t^{1-{\alpha}/{2}}},
\end{equation}
where $e^{-t(-\Delta+V)}$ denotes the semigroup generated by $-\Delta+V$. 
The Riesz potential $(-\Delta+V)^{-\alpha/2}$ has been extensively studied in the literature 
(see, e.g., \cite{y2008, qf2024, whz2026, yyz2009}).
In particular, Li, Rahm and Wick \cite[Theorem 1.3]{lrw2019} established the following 
quantitative estimate for this operator.

\medskip

\noindent{\bf Theorem D} (\cite{lrw2019}).
{\it 
Let $V\in RH^{n/2}$ and $\omega$ be a weight function on 
$\rn$. Suppose $0<\alpha<n$, $1<p<n/\alpha$ and $\theta>0$.
Then
\begin{equation*}
\left\| (-\Delta+V)^{-\frac{\alpha}{2}} \right\|_{L^p(\omega^p) \to L^q(\omega^q)} 
\lesssim 
[\omega]_{A_{p,q}^{\alpha,\theta/(3K)}}^{\left(1 - \frac{\alpha}{n}\right) 
\max\left\{1, \frac{p'}{q}\right\}},
\end{equation*}
where $q$ satisfies $1/q=1/p-\alpha/n$, $K$ is defined by the equation 
$( \frac{1}{K} + \frac{q}{K p'} ) ( 1 - \frac{\alpha}{n}) 
\max\{1, \frac{p'}{q}\} = \frac{1}{2}$,
\begin{equation*}
[\omega]_{A_{p,q}^{\alpha,\theta/(3K)}} := \sup_{ \mathrm{cube} \, Q\subseteq \rn}
\left( \frac{1}{\psi_{\theta/(3K)}(Q)|Q|} \int_Q w^q(x)\,dx \right)^{1/q}
\left( \frac{1}{\psi_{\theta/(3K)}(Q)|Q|} \int_Q w^{-p'}(x)\,dx \right)^{1/p'}
< \infty
\end{equation*}
and $\psi_{\theta/(3K)}(Q) := (1 + \frac{\ell(Q)}{\rho(x_Q,V)})^{\theta/(3K)}$,
with $x_Q$ the center of $Q$ and $\ell(Q)$ the side-length of $Q$.
}

\medskip

Our next result gives a characterization of the norm of Riesz potential $(-\Delta+V)^{-\alpha/2}$
in terms of perturbed dyadic cubes $\mathcal{D}^V$.

\begin{theorem}\label{p4.002}
Let $V\in RH^{n/2}$ satisfy Assumptions $\mathbf{(A_1)}$--$\mathbf{(A_3)}$
and $\mathcal{D}^{V}$ the perturbed dyadic cube system in Theorem \ref{t1.2}.
Then the following statements hold true.
\begin{enumerate}
\item[\textnormal{(i)}]
if $0<\alpha \leq 2$ and $1<p<\infty$, then
\begin{equation*}
\left\| (-\Delta+V)^{-\frac{\alpha}{2}} \right\|_{ L^p(dx) \to L^p(dx)}^p 
\sim
\sup_{Q\in\mathcal{D}^V} |Q|^{\frac{\alpha p'}{n}}.
\end{equation*}
If $0<\alpha<n$ and $1<p<\infty$, then
\begin{equation*}
\left\| (-\Delta+V)^{-\frac{\alpha}{2}} \right\|_{L^p(dx) \to L^p(dx)}^p 
\lesssim 
\sup_{Q\in\mathcal{D}^V} |Q|^{\frac{\alpha p'}{n}}.
\end{equation*}
\item[\textnormal{(ii)}] Let $\mu$ be a nonnegative locally finite Borel measure on $\mathbb{R}^n$.
If $0<\alpha \leq 2$ and $p>n/\alpha$, then
\begin{equation*}
\left\| (-\Delta+V)^{-\frac{\alpha}{2}} \right\|_{L^p(dx) \to L^p(d\mu)}^p 
\sim
\sup_{Q\in\mathcal{D}^V} |Q|^{\frac{\alpha p'}{n}-(p'-1)} \mu(Q)^{p'-1}.
\end{equation*}
If $0<\alpha<n$ and $p>n/\alpha$, then
\begin{equation*}
\left\| (-\Delta+V)^{-\frac{\alpha}{2}} \right\|_{L^p(dx) \to L^p(d\mu)}^p 
\lesssim 
\sup_{Q\in\mathcal{D}^V} |Q|^{\frac{\alpha p'}{n}-(p'-1)} \mu(Q)^{p'-1}.
\end{equation*}
\end{enumerate}
\end{theorem}

Theorem \ref{p4.002} will be proved in Section \ref{s4}. A key idea in the proof
is a perturbed sparse domination technique dapted to $\mathcal{D}^V$.
As an application of Theorem \ref{p4.002}, we obtain
the following quantitative spectral estimates for Schr\"odinger operators. 

\begin{corollary}\label{t1.3}
Let $L= -\Delta + V$ be a Schr\"odinger operator on $\rn$  with 
$V\in RH^{n/2}$ satisfying Assumptions $\mathbf{(A_1)}$–$\mathbf{(A_3)}$.
Suppose $\lambda_1(L):=\inf \sigma (L)$ is the ground state of $L$. 
For any $\lambda>\lambda_1(L)$,  denote by $N(\lambda,L):={\rm dim}\,{\rm ran\,}
\mathbf{1}_{(-\infty,\lambda]}(L)$ the rank of the spectral projection $\mathbf{1}_{(-\infty,\lambda]}(L)$. 
Then there exist constants $c, C> 0$ and $k\in\mathbb{Z}$  such that the following two assertions hold.
\begin{itemize}
\item [{\rm (i)}]  {(Spectrum bottom)} 
\begin{equation*}
c \inf\limits_{Q\in \mathcal{D}^V} |Q|^{-\frac{2}{n}}\le \lambda_1(L)\le C \inf\limits_{Q\in \mathcal{D}^V} |Q|^{-\frac{2}{n}},
\end{equation*}
where $\mathcal{D}^{V}$ is the perturbed dyadic cube system defined in Theorem \ref{t1.2}.
\item [{\rm (ii)}] {(Rank of spectral projection)} 
\begin{equation*}
c\,\#\left\{Q\in \mathcal{D}_0^{V+\lambda}:\  l(Q)=2^{k-[\log_2 \lambda/2]} \right\}\le 
N(\lambda, L) \le C\, \#\left\{Q\in \mathcal{D}_0^{V+\lambda}:\  l(Q)=2^{k-[\log_2 \lambda/2]}\right\},
\end{equation*} 
where $\mathcal{D}_0^{V+\lambda}$ is $0$-th generation of  the perturbed dyadic cube system from Theorem \ref{t1.2}.
\end{itemize}
\end{corollary}

\begin{remark}\label{rem1.5x}
For any $\lambda\in (0,\infty)$, let
\begin{equation*}
V(\lambda, L)
:=\left| \left\{x\in \mathbb{R}^n : \rho(x,V) >  \lambda^{-\frac{1}{2}}  \right\} \right|.
\end{equation*}
It is a classical result (see, e.g., \cite{s1996, bfs2023}) that that there exist constants $c,C> 0$  such that 
\begin{equation}\label{eqn-E}
(c\lambda)^\frac{n}{2}V(c\lambda, L) \leq N(\lambda, L) \leq (C\lambda)^\frac{n}{2}V(C\lambda, L).
\end{equation} 
Corollary \ref{t1.3} provides an explanation of \eqref{eqn-E}
by showing that the estimates 
of $N(\lambda, L)$  can be reduced to counting the number of  uniformly sized dyadic cubes 
in the spectrum-dominated region
(see the blue regions in Figures 3 and 4 for the latter region in some specific cases.)\

\begin{figure}[H]
  \centering  
  \begin{subfigure}[b]{0.45\textwidth}
    \centering
    \captionsetup{labelformat=empty}  
    \includegraphics[scale=0.3]{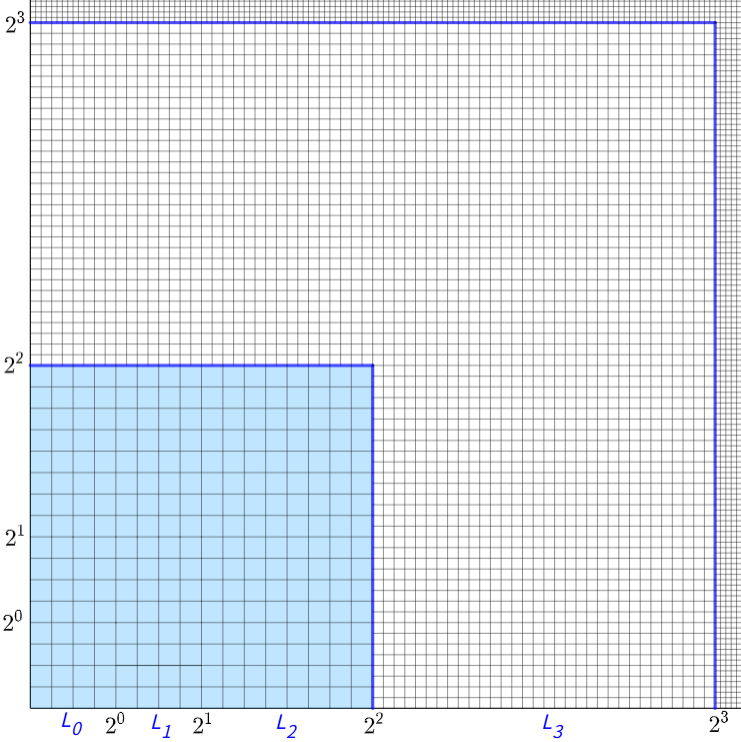}
    \caption{Figure 3: Dyadic cubes $\mathcal{D}^{V+\lambda}$ for $V(x)=|x|^2$} 
  \end{subfigure}
  \hfill 
  \begin{subfigure}[b]{0.45\textwidth}
    \centering
    \captionsetup{labelformat=empty}  
    \includegraphics[scale=0.28]{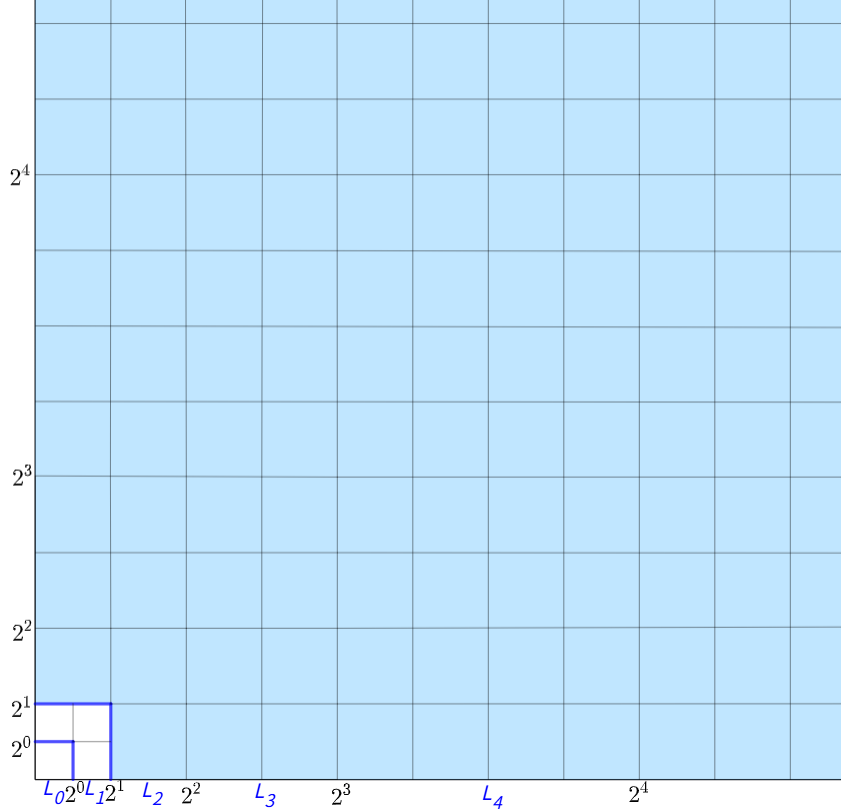}
    \caption{Figure 4: Dyadic cubes $\mathcal{D}^{V+\lambda}$ for $V(x)=|x|^{-1}$} 
  \end{subfigure}
\end{figure}
\end{remark}

The rest of this paper is organized as follows. In Section \ref{s2}, 
we collect some basic properties on the critical radius $\rho(x,V)$ for potential 
$V\in RH^q$ satisfying Assumptions $\mathbf{(A_1)}$–$\mathbf{(A_3)}$,
which play  essential roles in the construction of perturbed dyadic cubes $\mathcal{D}^V$. 
Section \ref{s3} is devoted to the construction of $\mathcal{D}^V$. 
We first prove Theorem \ref{t1.2} in Section \ref{s3.1}; then in Section \ref{s3.2}
we establish some basic properties of $\mathcal{D}^V$; finally, in Section \ref{s3.3}
we prove an adjacent property for  $\mathcal{D}^V$ by 
constructing an adjacent family of  perturbed dyadic cubes. 
Section \ref{s4} is devoted to quantitative estimates for the operator norms of the Riesz potential $(-\Delta+V)^{-\alpha/2}$ and their applications to spectral quantization.
We first show in Section \ref{s4.1} that the Riesz potential of $-\Delta+V$ 
can be bounded by a localized Riesz potential adapted to $\mathcal{D}^V$, 
then in Section \ref{s4.2}, we establish a perturbed sparse domination of this localized Riesz potential.
Based on this sparse domination, we establish Theorem \ref{p4.002} in Section \ref{s4.3},
and conclude with the proof of Corollary \ref{t1.3} in Section \ref{s4.4}.

We end this section by making some conventions on the notation. 
Let $\mathbb{N} = \{1,2,\ldots\}$, $\mathbb{Z}_+ = \mathbb{N} \cup \{0\}$, 
and $\mathbb{Z} = \{0,\pm 1, \pm 2, \ldots\}$. 
Let $C$ be a positive constant which is independent of the main parameters, 
but it may vary from line to line. If $f\leq C g$, we write 
$f\lesssim g$, if $f\lesssim g\lesssim f$, we then write $f\sim g$. 
For any subset $E\subseteq \rn$, $\chi_E$ denotes its characteristic function. 
Let $B(x,r) \subseteq \rn$ be a ball with center $x$ and radius $r$.
Let $Q(x,l)\subseteq \rn$ be a cube with center $x$, side length $l(Q)$ and 
edges parallel to the coordinate axes. 




\section{Preliminaries on critical radius}\label{s2}

This section is devoted to some basic properties on the critical radius function $\rho(x,V)$
for $V\in RH^{n/2}$ satsifying Assumptions $\mathbf{(A_1)}$–$\mathbf{(A_3)}$.
We begin with the following doubling property of $V$ from \cite[p.\,518]{s1995}.

\begin{lemma}[\cite{s1995}]\label{l2.1} 
Let $V\in RH^q$ with $q>1$. Then there exists $C_0>1$ such that, for any $x\in \rn$ and $r\in (0,\infty)$,
\begin{equation*}
\int_{B(x, 2r)} {V(y)\, d y} \leq C_0 \int_{B(x, r)} {V(y) \,d y}.
\end{equation*}
Moreover, for any $x\in \rn$ and $0<r<R<\infty$, it holds
\begin{equation*}
\int_{B(x, R)} {V(y) \,d y} \leq C_0 \left( \frac{R}{r} \right)^{\mathrm{log}_2 {C_0}} \int_{B(x, r)} {V(y) \,d y}.
\end{equation*}
\end{lemma}

\begin{lemma}[\cite{s1995}]\label{l2.3} 
Let $V\in RH^q$ with $q>1$. Then there exists $C>1$ such that, for any $0<r<R<\infty$,
\begin{equation*}
\frac{1}{r^{n-2}}\int_{B(x,r)} {V(y)\,dy}\leq 
C \left(\frac{R}{r}\right)^{\frac{n}{q}-2} \frac{1}{R^{n-2}}\int_{B(x,R)} {V(y)\,dy}.
\end{equation*}
\end{lemma}

For any $r>0$ and $x\in \rn$, let
\begin{equation*}
F(r,x,V):= \frac{1}{r^{n-2}} \int_{B(x, r)} V(y)\, d y.
\end{equation*}
Lemma \ref{l2.3}  shows  that $F(r,x,V)$ is quasi-increasing at $r$ when $V\in RH^{ n/2}$.
Based on this property, we have the following basic properties 
on $\rho(x,V)$.

\begin{lemma}\label{l2.4} 
Let $V\in RH^q$ with $q\ge n/2$. For any $r>0$ and $x\in\rn$. The 
following assertions hold.
\begin{enumerate}
\item [{\rm (i)}] For all $x,y\in \rn$ satisfying $|x-y|\lesssim \rho(x,V)$, 
it holds $\rho(x,V)\sim \rho(y,V)$.
\item[\textnormal{(ii)}]
If $r=\rho(x,V)$, then $F(r,x,V)=1$.
\item[\textnormal{(iii)}]
If $F(r,x,V)\gtrsim 1$, then $r\gtrsim \rho(x,V)$.
\item[\textnormal{(iv)}]
If $F(r,x,V)\lesssim 1$, then $r\lesssim \rho(x,V)$.
\end{enumerate}
\end{lemma}

\begin{proof} 
Since (i) and (ii) can be found in \cite{s1995}, we only prove (iii) and (iv). 

For (iii), let $F(r,x,V)\gtrsim 1$. 
If $r\geq \rho(x,V)$, then (iii) holds.
If $r< \rho(x,V)$, then let $R:= \rho(x,V)$. By Lemmas \ref{l2.3} and (ii), we have
\begin{align*}
F(r,x,V)
&= \frac{1}{r^{n-2}}\int_{B(x,r)} {V(y)\,dy} 
\leq C \left(\frac{R}{r}\right)^{\frac{n}{q}-2} \frac{1}{R^{n-2}}\int_{B(x,R)} {V(y)\,dy} \\
&= C \left(\frac{R}{r}\right)^{\frac{n}{q}-2} F(R,x,V)
= C \left(\frac{R}{r}\right)^{\frac{n}{q}-2}
= C \left(\frac{\rho(x,V)}{r}\right)^{\frac{n}{q}-2}.
\end{align*}
By this and the fact of $F(r,x,V)\gtrsim 1$, we have
\begin{equation}\label{e1-l2.4} 
C \left(\frac{\rho(x,V)}{r}\right)^{\frac{n}{q}-2}\geq F(r,x,V) \gtrsim 1,
\end{equation}
which combined with the fact $q>n/2$ implies $r\gtrsim \rho(x,V)$
and hence shows (iii). The proof of (iv) is similar, the details being omitted.
\end{proof}

We next relate the growth of $\rho(x, V)$ to Assumption $\mathbf{(A_1)}$ via the following lemma.

\begin{lemma}\label{l2.6} 
Let $V\in RH^{n/2}$. Then there exist $M_0>1$ such that for all 
$x\in\mathbb{R}^n$ satisfying $|x|\geq r_0>0$, the inequality
\begin{equation*}
\rho(x,V)\leq M_0 |x|
\end{equation*}
holds if and only if $V$ satisfies Assumption $\mathbf{(A_1)}$.
\end{lemma}

\begin{proof} 
By (iii) of Lemma \ref{l2.4}, we can easily prove the sufficiency. Next, we prove the necessity.
Supposed that there exist $M_0>1$ such that, 
for all $x\in\mathbb{R}^n$ satisfying $|x|\geq r_0>0$, 
$\rho(x,V)\leq M_0 |x|$.
Let $r:=\rho(x,V)$. Then $r<(M_0+1) |x|$. Note that $V\in RH^q$ with $q\ge n/2$.
By Lemmas \ref{l2.1}, \ref{l2.3} and \ref{l2.4}, we have
\begin{align*}
\frac{1}{r^{n-2}}\int_{B(x,r)} {V(y)\,dy} 
&\leq C \left(\frac{(M_0+1) |x|}{r}\right)^{\frac{n}{q}-2} 
      \frac{1}{((M_0+1)|x|)^{n-2}}\int_{B(x,(M_0+1) |x|)} {V(y)\,dy} \\
&\leq C \frac{1}{((M_0+1)|x|)^{n-2}}\int_{B(x,(M_0+1) |x|)} {V(y)\,dy} \\
&\leq C_0 \left( \frac{(M_0+1) |x|}{|x|} \right)^{\mathrm{log}_2 {C_0}} C (M_0+1)^{2-n} 
      \frac{1}{|x|^{n-2}} \int_{B(x,|x|)} {V(y)\,dy} \\
&\sim  \frac{1}{|x|^{n-2}} \int_{B(x,|x|)} {V(y)\,dy}.    
\end{align*}
By this and (ii) of Lemma \ref{l2.4}, we obtain
\begin{equation*}
\frac{1}{|x|^{n-2}} \int_{B(x,|x|)} {V(y)\,dy}\gtrsim \frac{1}{r^{n-2}}\int_{B(x,r)} {V(y)\,dy} \sim 1.
\end{equation*}
Therefore, $V$ satisfies Assumption $\mathbf{(A_1)}$.
\end{proof}

The following lemma establishes that the critical radius function $\rho(x, V)$ 
has the reverse doubling property over annular regions 
for potentials $V$ satisfying Assumptions 
$\mathbf{(A_1)}$ and $\mathbf{(A_2)}$.

\begin{lemma}\label{l2.7} 
Let $V\in RH^{n/2}$. If $V$ satisfies Assumptions 
$\mathbf{(A_1)}$ and $\mathbf{(A_2)}$, 
then there exists $\widetilde{C_D}>0$ such that, 
for all $r_1\geq r_0>0$ and $x,y\in\mathbb{R}^n$ satisfying 
$r_1\leq |x|\leq|y|< 2r_1$, it holds
\begin{equation*}
\rho(x,V)\leq \widetilde{C_D} \, \rho(y,V).
\end{equation*}
\end{lemma}

\begin{proof} 
For all $r_1\geq r_0>0$ and $x,y\in\mathbb{R}^n$ 
satisfying $r_1\leq |x|\leq|y|< 2r_1$,
there exists a rotational transformation $P$ such that
$P\left(|y|x/|x|\right)=y$.
Therefore, for any $r>0$, we have
\begin{align}\label{e1-l2.7}
\frac{1}{r^{n-2}}\int_{B(y,r)} V(z)\,dz
&=\frac{1}{r^{n-2}}\int_{|z-P\left(\frac{|y|}{|x|}x\right)|<r} V(z)\,dz \\ \notag
&=\frac{1}{r^{n-2}}\int_{\left|P(\widetilde{z}\,)-P\left(\frac{|y|}{|x|}x\right)\right|<r} V(P(\widetilde{z}\,))\,d(P(\widetilde{z}\,)) \\ \notag
&=\frac{1}{r^{n-2}}\int_{\left|\frac{|x|}{|y|}\widetilde{z}-x\right|<\frac{|x|}{|y|}r} V(P(\widetilde{z}\,))\,d\widetilde{z} \\ \notag
&=\frac{1}{r^{n-2}}\int_{\left|w-x\right|<\frac{|x|}{|y|}r} V\left(P\left(\frac{|y|}{|x|}w\right)\right)d\left(\frac{|y|}{|x|}w\right).
\end{align}
If $V$ satisfies Assumption $\mathbf{(A_1)}$, 
then by Lemma \ref{l2.6}, we obtain that 
there exist $M_0>1$ such that
\begin{equation}\label{e2-l2.7}
\rho(x,V)\leq M_0 |x|.
\end{equation}
We choose $r=\rho(x,V)|y|/(2M_0 |x|)$.
Then for any $w\in B\left(x,|x|r/|y|\right)$, 
by \eqref{e2-l2.7}, we have
$
|w-x|<|x|r/|y|=\rho(x,V)/(2M_0)\leq |x|/2.
$
By this and the fact that $r_0\leq r_1\leq |x|\leq|y|< 2r_1$, we obtain
$
|w|\geq \left||x|-|w-x|\right|=|x|-|w-x|\geq |x|/2 \geq r_0/2
$
and 
\begin{equation*}
\frac{1}{2}r_0\leq |w|\leq \left|P\left(\frac{|y|}{|x|}w\right)\right|<2|w|.
\end{equation*}
Therefore, if $V$ satisfies Assumption $\mathbf{(A_2)}$, then we have
\begin{equation}\label{e3-l2.7}
V\left(P\left(\frac{|y|}{|x|}w\right)\right)\leq C_D V(w).
\end{equation}
Note that $r=\rho(x,V)|y|/(2M_0 |x|)$.
By \eqref{e1-l2.7}, \eqref{e3-l2.7} and (ii) of Lemma \ref{l2.4}, we have
\begin{align*}
\frac{1}{r^{n-2}}\int_{B(y,r)} {V(z)\,dz}
&\leq C_D\frac{1}{r^{n-2}}\int_{\left|w-x\right|<\frac{|x|}{|y|}r}{V(w)\,d\left(\frac{|y|}{|x|}w\right)}\\
&= C_D\left(\frac{|y|}{2M_0 |x|}\right)^{2-n}\frac{1}{\rho(x,V)^{n-2}}
\int_{\left|w-x\right|<\frac{\rho(x,V)}{2M_0}} {V(w)\,d\left(\frac{|y|}{|x|}w\right)}\\
&\leq C_D(2M_0)^{n-2}\left(\frac{|y|}{|x|}\right)^{2}
\frac{1}{\rho(x,V)^{n-2}}\int_{\left|w-x\right|<\rho(x,V)} {V(w)\,dw}\\
&= C_D(2M_0)^{n-2}\left(\frac{|y|}{|x|}\right)^{2}
\leq 2^n {M_0}^{n-2} C_D \lesssim 1.           
\end{align*}
By this and (iv) of Lemma \ref{l2.4}, we obtain that there exists a constant $\overline{C}_D>0$ depending on $C_D$ such that $r\leq \overline{C}_D \rho(y,V)$, namely,
$\rho(x,V)|y|/(2M_0 |x|)\leq \overline{C}_D \rho(y,V)$.
Therefore,
\begin{equation*}
\rho(x,V)\leq  2M_0 \frac{|x|}{|y|} \overline{C}_D \rho(y,V)\leq  2M_0 \overline{C}_D \rho(y,V),
\end{equation*}
which completes the proof of Lemma \ref{l2.7}.
\end{proof}

The following lemma establishes that the critical radius function $\rho(x, V)$ has the doubling property over annular regions for potentials $V$ satisfying Assumptions $\mathbf{(A_1)}$ and $\mathbf{(A_3)}$.

\begin{lemma}\label{l2.8} 
Let $V\in RH^{n/2}$. If $V$ satisfies Assumptions $\mathbf{(A_1)}$ and $\mathbf{(A_3)}$, 
then there exists $\widetilde{C_{RD}}>0$ such that, for all $r_1\geq r_0>0$ 
and $x,y\in\mathbb{R}^n$ satisfying $r_1\leq |x|\leq|y|< 2r_1$, 
it holds
\begin{equation*}
\rho(y,V)\leq \widetilde{C_{RD}} \, \rho(x,V).
\end{equation*}
\end{lemma}

\begin{proof} 
For all $r_1\geq r_0>0$ and $x,y\in\mathbb{R}^n$ satisfying $r_1\leq |x|\leq|y|< 2r_1$,
there exists a rotational transformation $P$ such that
$P\left(|x|y/|y|\right)=x$.
Therefore, for any $r>0$, we have
\begin{align}\label{e1-l2.8}
\frac{1}{r^{n-2}}\int_{B(x,r)} {V(z)\,dz} 
&=\frac{1}{r^{n-2}}\int_{\left|z-P\left(\frac{|x|}{|y|}y\right)\right|<r} {V(z)\,dz}\\ \notag
&=\frac{1}{r^{n-2}}\int_{\left|P(\,\widetilde{z}\,)-P\left(\frac{|x|}{|y|}y\right)\right|<r} 
  {V(P(\,\widetilde{z}\,))\,d(P(\,\widetilde{z}\,))} \\ \notag
&=\frac{1}{r^{n-2}}\int_{\left|\frac{|y|}{|x|}\widetilde{z}-y\right|<\frac{|y|}{|x|}r} 
  {V(P(\,\widetilde{z}\,))\,d\widetilde{z}} \\ \notag
&=\frac{1}{r^{n-2}}\int_{\left|w-y\right|<\frac{|y|}{|x|}r} 
  {V\left(P\left(\frac{|x|}{|y|}w\right)\right)\,d\left(\frac{|x|}{|y|}w\right)}.
\end{align}
If $V$ satisfies Assumption $\mathbf{(A_1)}$, then by Lemma \ref{l2.6}, 
we obtain that there exist $M_0>1$ such that
\begin{equation}\label{e2-l2.8}
\rho(y,V)\leq M_0 |y|.
\end{equation}
We choose $r=\rho(y,V)|x|/(2M_0 |y|)$.
Then for any $w\in B\left(y,|y|r/|x|\right)$, by \eqref{e2-l2.8}, 
we have $|w-y|<|y|r/|x|=\rho(y,V)/(2M_0)\leq |y|/2$.
This implies that
$
|w|\geq \left||y|-|w-y|\right|=|y|-|w-y|\geq |y|/2.
$
By this and the fact that $r_0\leq r_1\leq |x|\leq|y|< 2r_1$, we obtain
\begin{equation*}
\frac{1}{2}r_0\leq \frac{1}{2}|x|\leq \left|P\left(\frac{|x|}{|y|}w\right)\right|
\leq |w| < 2 \left|P\left(\frac{|x|}{|y|}w\right)\right|.
\end{equation*}
Therefore, if $V$ satisfies Assumption $\mathbf{(A_3)}$, then we have
\begin{equation}\label{e3-l2.8}
V\left(P\left(\frac{|x|}{|y|}w\right)\right)\leq C_{RD} V(w).
\end{equation}
Note that $r=\rho(y,V)|x|/(2M_0 |y|)$.
By \eqref{e1-l2.8}, \eqref{e3-l2.8} and (ii) of Lemma \ref{l2.4}, we have
\begin{align*}
\frac{1}{r^{n-2}}\int_{B(x,r)} {V(z)\,dz}
&= \frac{1}{r^{n-2}}\int_{\left|w-y \right|<\frac{|y|}{|x|}r}
{V\left(P\left(\frac{|x|}{|y|}w\right)\right)d\left(\frac{|x|}{|y|}w\right)}\\
&\leq C_{RD}\frac{1}{r^{n-2}}\int_{\left|w-y\right|<\frac{|y|}{|x|}r}{V(w)d\left(\frac{|x|}{|y|}w\right)}\\
&= C_{RD}\left(\frac{|x|}{2M_0 |y|}\right)^{2-n}\frac{1}{\rho(y,V)^{n-2}}
\int_{\left|w-y \right|<\frac{\rho(y,V)}{2M_0}} {V(w)d\left(\frac{|x|}{|y|}w\right)}\\
&\leq C_{RD}(2M_0)^{n-2}\left(\frac{|x|}{|y|}\right)^{2}
\frac{1}{\rho(y,V)^{n-2}}\int_{\left|w-y\right|<\rho(y,V)} {V(w)dw}\\
&= C_{RD}(2M_0)^{n-2}\left(\frac{|x|}{|y|}\right)^{2}
\leq C_{RD}(2M_0)^{n-2} \lesssim 1.           
\end{align*}
By this and (iv) of Lemma \ref{l2.4}, we obtain that 
there exists a constant $\overline{C}_{RD}>0$ depending on $C_{RD}$ 
such that $r\leq \overline{C}_{RD} \, \rho(x,V)$, namely,
$
\rho(y,V)|x|/(2M_0 |y|)\leq \overline{C}_{RD} \, \rho(x,V).
$
Therefore,
\begin{equation*}
\rho(y,V)\leq  
2M_0 \frac{|y|}{|x|} \overline{C}_{RD} \, \rho(x,V)
\leq  4M_0 \overline{C}_{RD} \, \rho(x,V),
\end{equation*}
which completes the proof of Lemma \ref{l2.8}.
\end{proof}

Let $\lambda>0$.
The following lemma gives the relation between 
$\rho(x, V+\lambda)$ and $\rho(x, V)$, which will play an important role 
in the construction of the perturbed dyadic cubes $\mathcal{D}^{V+\lambda}$ when 
proving Corollary \ref{t1.3}.

\begin{lemma}\label{l2.9} 
Let $V\in RH^{n/2}$ and $\lambda>0$. Then for any $x\in \rn$,
\begin{equation*}
\rho(x,V+\lambda)\sim \min\{\rho(x,V), \rho(x,\lambda)\}.
\end{equation*}
\end{lemma}

\begin{proof} 
First, we prove that 
\begin{equation}\label{e1-l2.9} 
\rho(x,V+\lambda)\gtrsim \min\{\rho(x,V), \rho(x,\lambda)\}.
\end{equation}

If $\min\{\rho(x,V), \rho(x,\lambda)\}=\rho(x,V)$, then $r_1:=\rho(x,V)\leq \rho(x,\lambda)\sim \frac{1}{\sqrt{\lambda}}$, 
and hence by (ii) of Lemma \ref{l2.4}, we have
\begin{equation*}
\frac{1}{r_1^{n-2}} \int_{B(x,r)} {(V(y)+\lambda)\,dy}
=\frac{1}{r_1^{n-2}} \int_{B(x,r)} {V(y)\,dy} + r_1^2 \lambda 
=1+ r_1^2 \lambda \lesssim 1.
\end{equation*}
From this and (iv) of Lemma \ref{l2.4}, we deduce that $r_1=\rho(x,V)\lesssim \rho(x,V+\lambda)$.

If $\min\{\rho(x,V), \rho(x,\lambda)\}=\rho(x,\lambda)$, then $r_2:=\rho(x,\lambda)< \rho(x,V)$, 
and hence by Lemma \ref{l2.3} and (ii) of Lemma \ref{l2.4}, we have
\begin{align*}
\frac{1}{r_2^{n-2}} \int_{B(x,r)} {(V(y)+\lambda)\,dy} 
&=\frac{1}{r_2^{n-2}} \int_{B(x,r)} {V(y)\,dy} + r_2^2 \lambda \\
&\lesssim \left( \frac{\rho(x,V)}{r_2} \right)^{\frac{n}{q}-2} 
\frac{1}{\rho(x,V)^{n-2}} \int_{B(x,\rho(x,V))} {V(y)\,dy} + r_2^2 \lambda \\
&\lesssim \frac{1}{\rho(x,V)^{n-2}} \int_{B(x,\rho(x,V))} {V(y)\,dy} + r_2^2 \lambda \\
&\sim 1 + r_2^2 \lambda \lesssim 1.
\end{align*}
From this and (iii) of Lemma \ref{l2.4}, we deduce that $r_2=\rho(x,\lambda)\lesssim \rho(x,V+\lambda)$.

Therefore, \eqref{e1-l2.9} holds. The reverse inequality follows directly from the definition of the critical radius function $\rho$ as in \eqref{eqn-rho}.
\end{proof}



\section{Perturbed dyadic cubes for Schr\"odinger operators}\label{s3}

This section is devoted to the construction of the perturbed dyadic cube system $\mathcal{D}^V$ for
$V \in RH^{n/2}$ satisfying Assumptions $\mathbf{(A_1)}$–$\mathbf{(A_3)}$. 
We first prove Theorem~\ref{t1.2} in Subsection~\ref{s3.1}, then establish some basic properties of $\mathcal{D}^V$ 
in Subsection~\ref{s3.2}, and finally prove an adjacent property for $\mathcal{D}^V$ in Subsection~\ref{s3.3} 
by constructing an adjacent family of perturbed dyadic cubes.  

\subsection{Construction of perturbed dyadic cubes $\mathcal{D}^V$}\label{s3.1}

Let $V \in RH^{n/2}$ satisfy Assumptions $\mathbf{(A_1)}$–$\mathbf{(A_3)}$. 
In this subsection, we construct the perturbed dyadic cube system $\mathcal{D}^V$ which 
proves Theorem \ref{t1.2}. 

First, we define some related concepts as follows.
\begin{itemize}
\item {\bf (Dyadic scale)}: 
Let $N_0\in \mathbb{Z}_+$ satisfy 
$2^{N_{0}-1}<r_0\leq 2^{N_0}$ with $r_0$ being
the constant defined in Assumption $\mathbf{(A_1)}$. 
Define the lattice points by 
\begin{equation}\label{eqn-LP}
a_0:=2^{N_0} \quad \mathrm{and} \quad a_l:=2a_{l-1} \quad 
\mathrm{for} \ \mathrm{any} \ l\in \mathbb{N}.
\end{equation}

\item {\bf (Layers)}: 
For any $l\in \mathbb{N}$, define the layers by
\begin{equation}\label{eqn-Lay}
L_0:=[-a_0,a_0)^n \quad \mathrm{and} \quad 
L_l:= \left[-a_l,a_l\right)^n \backslash \left[-a_{l-1},a_{l-1}\right)^n.
\end{equation}

\item {\bf (Layer scale)}: 
For any $l\in \mathbb{N}$, 
let $x_l:= (a_{l-1},0,\ldots,0)\in \rn$ 
be the lattice point and $k_l \in \mathbb{Z}$ satisfy
$
2^{-k_l}\leq \rho(x_l,V)/M_0<2^{-(k_l-1)},
$
where $M_0$ is the constant defined in Lemma \ref{l2.6}. 
Define the layer scale size
\begin{equation}\label{eqn-LSS}
b_0:=2^{N_{0}} \quad \mathrm{and} \quad b_l:= 2^{-k_l} \quad 
\mathrm{for} \ \mathrm{any} \ l\in \mathbb{N}.
\end{equation}
\end{itemize}

Next, we define the perturbed dyadic cube system $\mathcal{D}^V$.

\begin{definition}\label{d3.1}
Let $V\in RH^{n/2}$ satisfy Assumptions $\mathbf{(A_1)}$–$\mathbf{(A_3)}$.
For any $k,l\in \mathbb{Z}_+$, define
\begin{equation*}
\mathcal{D}_{k,l}^{V}:= \{Q\in \mathcal{D}: Q\subseteq L_l, \, l(Q)=2^{-k}b_l \},
\end{equation*}
where $\mathcal{D}$ denotes the class dyadic cube system as in \eqref{eqn-CDC}, 
$L_l$ is the $l$-th layer as in \eqref{eqn-Lay} and 
$b_l>0$ the layer scale as in \eqref{eqn-LSS}. 
Then the perturbed dyadic cube system $\mathcal{D}^V$ is defined by
\begin{equation*}
\mathcal{D}_{k}^{V}:=\bigcup\limits_{l\in \mathbb{Z}_+} \mathcal{D}_{k,l}^{V} \quad \text{and}
\quad
\mathcal{D}^{V}:=\bigcup\limits_{k\in \mathbb{Z}_+} \mathcal{D}_{k}^{V}.
\end{equation*}
\end{definition}

To prove Theorem \ref{t1.2}, it suffices to show the following result.

\begin{proposition}\label{p3.10}
Let $V\in RH^{n/2}$ satisfy Assumptions $\mathbf{(A_1)}$–$\mathbf{(A_3)}$.
Then the perturbed dyadic cube system $\mathcal{D}^V$ defined as in 
Definition \ref{d3.1} is well-defined and satisfies all the properties 
(i)-(iii)  stated in Theorem \ref{t1.2}.
\end{proposition}

\begin{proof}   
We first prove that the perturbed dyadic system $\mathcal{D}^V$  is well-defined. To this end,
it suffices to show that for any $k, l\in \mathbb{Z}_+$, 
\begin{align}\label{eqn-wd}
\#\left\{Q\in \mathcal{D}: Q\subseteq L_l, \, l(Q)=2^{-k}b_l \right\}\in\mathbb{N}
\end{align}
For $l=0$, let $\widehat{M}_{k,0}:=a_0/(2^{-k}b_0)$.
Then by the definitions of $a_0$ and $b_0$ 
(see \eqref{eqn-LP} and \eqref{eqn-LSS}),
we have $\widehat{M}_{k,0}=2^{k}\in \mathbb{N}$. 
For $l\in \mathbb{N}$, let
\begin{equation*}
M_{k,l}:=\frac{a_{l-1}}{2^{-k}b_l}, \quad 
\overline{M}_{k,l}:= \frac{a_{l}-a_{l-1}}{2^{-k}b_l}.
\end{equation*}

For $M_{k,l}$, let $x_l:= (a_{l-1},0,\ldots,0)\in \rn$. 
Since $V$ satisfies Assumption $\mathbf{(A_1)}$, 
then by Lemma \ref{l2.6}, it follows that 
\begin{equation*}
\rho(x_l,V)\leq M_0 |x_l|.
\end{equation*}
By this and the definitions of $a_{l-1}$ and $b_l$ 
(see \eqref{eqn-LP} and \eqref{eqn-LSS}), we have
$
b_l=2^{-k_l}\leq \rho(x_l,V)/M_0\leq |x_l|=a_{l-1}=2^{N_0 +l-1},
$
where $N_0\in \mathbb{N}$, $k_l \in \mathbb{Z}$ are constants. 
From this, we deduce that
$
a_{l-1}/b_l=2^{N_0 +l-1+k_l}\geq 1.
$
Then $N_0 +l-1+k_l\geq 0$, and hence
\begin{equation}\label{e1-t1.1}
M_{k,l}=\frac{a_{l-1}}{2^{-k}b_l}=2^{N_0 +l-1+k_l+k}\in \mathbb{N}.
\end{equation}

For $\overline{M}_{k,l}$, by the definition of $a_l$ (see \eqref{eqn-LP}), we have
\begin{equation}\label{e2-t1.1}
\overline{M}_{k,l}= \frac{a_{l}-a_{l-1}}{2^{-k}b_l}=\frac{2a_{l-1}-a_{l-1}}{2^{-k}b_l}
=M_{k,l} \in \mathbb{N}.
\end{equation}
From \eqref{e1-t1.1} and \eqref{e2-t1.1}, we deduce that for any $k,l\in \mathbb{Z}_+$, 
\eqref{eqn-wd} holds true. This implies that  $\mathcal{D}^{V}$ is well-defined. 

Obviously, $\mathcal{D}^V$ satisfies properties (i) and (ii). 
This is because $\mathcal{D}_{k,l}^V$ is a subset of 
the classical dyadic cube system $\mathcal{D}$ 
(which inherently has nesting and partitioning properties), 
and the union of such subsets preserves these properties. 

Now, we prove that $\mathcal{D}^V$ satisfies property (iii). 
Namely, we need to prove that for any 
$k\in \mathbb{Z}_+$, any $Q\in \mathcal{D}_k^V$ and every $x\in Q$, 
we have $l(Q)\sim 2^{-k}\rho(x,V)$.

If $Q\in \mathcal{D}_{k,l}^V$ with $l=0$, then $l(Q)=2^{-k}b_0$. 
It suffices to prove that for any $x\in L_0$, 
\begin{equation}\label{e3-t1.1}
\rho(x,V)\sim b_0.
\end{equation}
For any $x\in B\left(0,\sqrt{n}a_0\right)\supseteq L_0=[-a_0, a_0)^n$, let $r:=2\sqrt{n}a_0$,
then we have
\begin{equation*}
B\left(0,\sqrt{n} a_0\right)\subseteq B(x,r)\subseteq  B\left(0,3 \sqrt{n} a_0\right).
\end{equation*}
By this and the fact that $V\in L_{loc} (\rn)$, we have
\begin{equation*}
\frac{1}{r^{n-2}}\int_{B(x,r)}{V(y)dy}\leq 
\frac{1}{(3 \sqrt{n} a_0)^{n-2}}\int_{B\left(0,3 \sqrt{n} a_0\right)}{V(y)dy}\sim 1
\end{equation*}
and
\begin{equation*}
\frac{1}{r^{n-2}}\int_{B(x,r)}{V(y)dy}\geq
\frac{1}{( \sqrt{n} a_0)^{n-2}}\int_{B\left(0,\sqrt{n} a_0\right)}{V(y)dy}\sim 1.
\end{equation*}
This, combined with (ii) and (iii) of Lemma \ref{l2.4} , can lead to
$
\rho(x,V)\sim r=2\sqrt{n}a_0\sim b_0,
$
and hence \eqref{e3-t1.1} holds.

If $Q\in \mathcal{D}_{k,l}^V$ with $l\in \mathbb{N}$, 
then $l(Q)=2^{-k}b_l$. 
It suffices to prove that for any $x\in L_l$, 
\begin{equation}\label{e4-t1.1}
\rho(x,V)\sim b_l.
\end{equation}
Let $x_l:= (a_{l-1},0,\ldots,0)\in \rn$. 
On the one hand, by the definition of $b_l$ (see \eqref{eqn-LSS}),
we have $b_l \sim \rho(x_l,V)$. 
On the other hand, since $V$ also satisfies 
Assumptions $\mathbf{(A_2)}$ and $\mathbf{(A_3)}$, 
by Lemmas \ref{l2.7} and \ref{l2.8}, 
we obtain that for any $x\in L_l$, 
$\rho(x,V) \sim \rho(x_l,V)$. 
Therefore, \eqref{e4-t1.1} holds.
This finishes the proof of Proposition \ref{p3.10}.
\end{proof}

\subsection{Some basic properties of $\mathcal{D}^V$}\label{s3.2}

In this subsection, we establish some basic properties of the 
perturbed dyadic cube system $\mathcal{D}^V$ constructed in Subsection \ref{s3.1}.
These properties play a key role in the perturbed sparse domination 
in Section \ref{s4}. We begin with the following result on the estimates for 
points in different layers.

\begin{lemma}\label{l3.11}
Let $V \in \mathrm{RH}^{n/2}$ satisfy Assumptions $\mathbf{(A_1)}$–$\mathbf{(A_3)}$.
Let $\{L_l\}_{l \in \mathbb{Z}_+}$ denote the layers as defined in Definition \ref{d3.1}.
Then for every $l \in \mathbb{Z}_+$ and each $x \in L_l$, the following assertions hold.
\begin{enumerate}
\item[\textup{(i)}]
For every integer $i \geq 2$ and all $y \in L_{l+i}$, there exists a constant $C_G \in (0,1)$ for which
\begin{equation*}
|y-x| \gtrsim 2^{(1-C_G)l} 2^i \rho(x,V).
\end{equation*}
\item[\textup{(ii)}]
For every integer $2 \leq i \leq l$ and all $y \in L_{l-i}$, there exists a constant $C_G \in (0,1)$ for which
\begin{equation*}
|y-x| \gtrsim 2^{(1-C_G)l} \rho(x,V).
\end{equation*}
\end{enumerate}
\end{lemma}

\begin{proof}   
For any $l \in \mathbb{Z}_+$, let $a_l$ and $b_l$ be defined as in 
\eqref{eqn-LP} and \eqref{eqn-LSS}, respectively. 
We first prove (i). To this end, 
we split the argument into two cases based on the value of $l$:

{\bf Case 1: $l=0$}. 
In this case, we have $x \in L_0 = [-a_0, a_0)^n$. 
For any integer $i \geq 2$ and $y \in L_{0+i}=L_i$, 
the layer definition (see \eqref{eqn-Lay}) gives 
$|y| \geq a_{i-1}=2^{N_0+i-1}$ and $|x| \leq a_0=2^{N_0}$, 
where $N_0$ is the fixed integer from \eqref{eqn-LP}.
By this and \eqref{e3-t1.1}, we obtain
\begin{equation*}
|y-x| \geq |y| - |x| \geq 2^{N_0+i-1} - 2^{N_0} = 2^{N_0}(2^{i-1}-1) \sim 2^i\sim 2^i \rho(x, V), 
\end{equation*}
and hence (i) holds.

{\bf Case 2: $l \geq 1$}. 
In this case, let $A_l := (a_{l-1}, 0, \dots, 0) \in \mathbb{R}^n$. 
By Lemma \ref{l2.6}, there exists a constant $M_0>1$ such that 
$\rho(A_l, V) \leq M_0 |A_l| = M_0 a_{l-1}$. 
Combining this with the definition of $a_{l-1}$ and $b_l$:
$
b_l = 2^{-k_l} \leq \rho(A_l, V)/M_0 \leq a_{l-1} = 2^{N_0+l-1},
$
where $k_l \in \mathbb{Z}$ is the index characterizing $b_l$ in \eqref{eqn-LSS}. 
Rearranging gives $-k_l \leq N_0+l-1 < N_0+l$. 
Since $N_0$ is independent of $l$, 
there exists a constant 
$C_G \in(0,1)$ such that:
\begin{equation}\label{e1-l3.11}
-k_l \leq C_G(N_0+l).
\end{equation}
For any integer $i \geq 2$ and $y \in L_{l+i}$, 
the layer definition implies $|y| \geq a_{l+i-1}$ and $|x| \leq a_l$. 
By this, \eqref{e1-l3.11} and \eqref{e4-t1.1}, we obtain
\begin{align*}
|y-x|
&\geq |y|-|x| \geq a_{l +i -1}-a_{l}
 = \frac{a_{l} (2^{i-1}-1)}{b_{l}}b_{l} \\
&=\frac{2^{N_0 + l} (2^{i-1}-1)}{2^{-k_{l}}}b_{l} 
 \geq \frac{2^{N_0 + l} (2^{i-1}-1)}{2^{C_G (N_0 +l)}}b_{l} \\
&\sim 2^{(1-C_G)l} 2^{i} \rho(x,V),
\end{align*}
and hence (i) holds.

We now prove (ii).
For any integer $2 \leq i \leq l$ and $y \in L_{l-i}$, 
the layer definition gives $|y| \leq a_{l-i}$ and $|x| \geq a_{l-1}$. 
By this, \eqref{e1-l3.11} and \eqref{e4-t1.1}, we obtain
\begin{align*}
|y-x|
&\geq |x| - |y|  \geq  a_{l-1}-a_{l -i}
 = \frac{a_{l-1}-a_{l -i}}{b_{l}}b_{l} \\
&=\frac{2^{N_0 + l -i} (2^{i-1}-1)}{2^{-k_{l}}}b_{l} 
 \geq \frac{2^{N_0 + l -i} (2^{i-1}-1)}{2^{C_G (N_0 +l)}}b_{l} \\
&\sim 2^{(1-C_G)l} \rho(x,V),
\end{align*}
and hence (ii) holds. 
This finishes the proof of Lemma \ref{l3.11}.
\end{proof}

The next lemma provides an upper bound for the number of 0-th generation cubes 
in each layer $L_l$.
\begin{lemma}\label{l3.12}
Let $V \in {RH}^{n/2}$ satisfy Assumptions $\mathbf{(A_1)}$–$\mathbf{(A_3)}$.
Suppose $\mathcal{D}_0^V = \bigcup_{l \in \mathbb{Z}_+} \mathcal{D}_{0,l}^V$ 
is the $0$-th generation of the perturbed dyadic cube system $\mathcal{D}^V$. 
For any $l\in \mathbb{Z}_+$, let $N(l)$ be the number of cubes 
in the set $\mathcal{D}^V_{0,l}$.
Then there exists $k_D \in \mathbb{Z}$ such that
\begin{equation*}
N(l)\lesssim 2^{(1- k_D)nl}.
\end{equation*}
\end{lemma}

\begin{proof}   
For any $l \in \mathbb{Z}_+$, let $a_l$, $L_l$, and $b_l$ 
be the dyadic scale, layers, and layer scale 
defined in \eqref{eqn-LP}, \eqref{eqn-Lay}, and \eqref{eqn-LSS}, respectively.
If $l = 0$, then $L_0 = [-a_0, a_0)^n$, $b_0 = 2^{N_0} = a_0$ and
\begin{equation*}
N(0) = \frac{|L_0|}{b_0^n} = \frac{(2a_0)^n}{b_0^n} = 2^n .
\end{equation*}
For $l \geq 1$, let $A_l := (a_{l-1}, 0, \dots, 0) \in \mathbb{R}^n$.
By \eqref{e3-t1.1}, \eqref{e4-t1.1}, and Lemma \ref{l2.7}, 
there exists $k_D \in \mathbb{Z}$ such that
$
b_0\sim \rho(0,V) \lesssim 2^{-k_{D}l} \rho(A_l,V)\sim 2^{-k_{D}l} b_{l}. 
$
By this, we have
\begin{align*}
N(l)
 &= \frac{|L_l|}{b_l^n}
 \sim \frac{a_{l}^n -a_{l-1}^n}{b_l^n} 
 \lesssim \frac{a_{l}^n -a_{l-1}^n}{\left(2^{k_{D}l}b_0\right)^n} \\
&\sim\frac{\left(2^{l}a_0\right)^n -\left(2^{l-1}a_0\right)^n}
          {\left(2^{k_{D}l}b_0\right)^n}
 \sim\frac{\left(2^{l-1}a_0\right)^n \left(2^n -1 \right)}
          {\left(2^{k_{D}l}b_0\right)^n} \\
&\sim 2^{(1-k_{D})nl}.
\end{align*}
This completes the proof of Lemma \ref{l3.12}.
\end{proof}

The lemma below delivers a key sum inequality for arbitrary subsets of $\mathcal{D}^V$.

\begin{lemma}\label{l4.83}
Let  $V\in RH^{n/2}$ satisfy Assumptions $\mathbf{(A_1)}$–$\mathbf{(A_3)}$.
Suppose $\widetilde{\mathcal{D}}^V\subset \mathcal{D}^V$ is an arbitrary subset of the perturbed dyadic cube 
system $\mathcal{D}^V$. Then for any $s \geqslant 1$, there exists a constant $C(s)$ such that for any 
family of positive real numbers $\{\lambda_Q\}_{Q \in \widetilde{\mathcal{D}}^V}$,
\begin{equation}\label{l4.83-e1}
\left( \sum_{Q \in \widetilde{\mathcal{D}}^{V}} \lambda_Q \right)^s \leqslant C(s) 
\sum_{Q \in \widetilde{\mathcal{D}}^{V}} \lambda_Q \left( \sum_{Q' \subset Q} \lambda_{Q'} \right)^{s-1}. 
\end{equation}
\end{lemma}

\begin{proof}   
We begin by proving the lemma for integer values $s = m \in \mathbb{Z}^+$ via mathematical induction.
The estimate \eqref{l4.83-e1} holds trivially for $m = 1$. Suppose the estimate \eqref{l4.83-e1} holds 
for some integer $m \geq 1$; we now verify its validity for $m+1$:
\begin{align*}
\left( \sum_{Q \in \widetilde{\mathcal{D}}^{V}} \lambda_Q \right)^{m+1} 
&\leqslant C(m) \left( \sum_{Q \in \widetilde{\mathcal{D}}^{V}} \lambda_Q \right) 
\left( \sum_{Q' \in \widetilde{\mathcal{D}}^{V}} \lambda_{Q'} \left( \sum_{Q'' \subset Q'} \lambda_{Q''} \right)^{m-1} \right) \\
&= C(m) \sum_{Q' \in \widetilde{\mathcal{D}}^{V}} \lambda_{Q'}  
\left( \sum_{Q \in \widetilde{\mathcal{D}}^{V}} \lambda_{Q} \right) 
\left( \sum_{Q'' \subset Q'} \lambda_{Q''} \right)^{m-1} \\
&= C(m) \sum_{Q' \in \widetilde{\mathcal{D}}^{V}} \lambda_{Q'} 
\left( \sum_{Q \subsetneq Q'} \lambda_Q + \sum_{Q' \subset Q} \lambda_Q \right) 
\left( \sum_{Q'' \subset Q'} \lambda_{Q''} \right)^{m-1} \\ 
&\leq C(m) \sum_{Q' \in \widetilde{\mathcal{D}}^{V}} \lambda_{Q'} \left( \sum_{Q'' \subset Q'} \lambda_{Q''} \right)^m 
 + C(m) \sum_{Q' \in \widetilde{\mathcal{D}}^{V}} \lambda_{Q'} \left( \sum_{Q' \subset Q} \lambda_Q \right) \left( \sum_{Q'' \subset Q'} \lambda_{Q''} \right)^{m-1}.
\end{align*}
We next analyze the second term in the final summation above. By virtue of the nested property of 
$\widetilde{\mathcal{D}}^V$ in  Theorem \ref{t1.2}, we obtain
\begin{align*}
&\sum_{Q' \in \widetilde{\mathcal{D}}^{V}} \lambda_{Q'} \sum_{Q' \subset Q} \lambda_Q  
\left( \sum_{Q'' \subset Q'} \lambda_{Q''} \right)^{m-1} \\
&\quad=
\sum_{Q \in \widetilde{\mathcal{D}}^{V}} \lambda_Q \left( \sum_{Q' \subset Q} \lambda_{Q'} \right) 
\left( \sum_{Q'' \subset Q'} \lambda_{Q''} \right)^{m-1}
\leq \sum_{Q \in \widetilde{\mathcal{D}}^{V}} \lambda_Q \left( \sum_{Q' \subset Q} \lambda_{Q'} \right)^m,
\end{align*}
which gives the lemma for all nonnegative integers $s = m \in \mathbb{Z}^+$. 
In general, let $s = m + \varepsilon$, $0 < \varepsilon < 1$. 
Using the integer case estimate for $m$, we know
\begin{align*}
\left( \sum_{Q \in \widetilde{\mathcal{D}}^{V}} \lambda_Q \right)^s 
&= \frac{\left( \sum_{Q \in \widetilde{\mathcal{D}}^{V}} \lambda_Q \right)^{m+1}}
{\left( \sum_{Q \in \widetilde{\mathcal{D}}^{V}} \lambda_Q \right)^{1 - \varepsilon}} 
\leqslant C(m) \sum_{Q \in \widetilde{\mathcal{D}}^{V}} \lambda_Q \frac{\left( \sum_{Q' \subset Q} 
\lambda_{Q'} \right)^m}{\left( \sum_{Q'' \in \widetilde{\mathcal{D}}^{V}} \lambda_{Q''} \right)^{1 - \varepsilon}} \\
&\leqslant C(m) \sum_{Q \in \widetilde{\mathcal{D}}^{V}} \lambda_Q \frac{\left( \sum_{Q' \subset Q} 
\lambda_{Q'} \right)^m}{\left( \sum_{Q'' \subset Q} \lambda_{Q''} \right)^{1 - \varepsilon}} 
= C(m) \sum_{Q \in \widetilde{\mathcal{D}}^{V}} \lambda_Q \left( \sum_{Q' \subset Q} \lambda_{Q'} \right)^{m+\varepsilon-1} \\
&= C(m) \sum_{Q \in \widetilde{\mathcal{D}}^{V}} \lambda_Q \left( \sum_{Q' \subset Q} \lambda_{Q'} \right)^{s-1}.   
\end{align*}
This completes the proof of Lemma \ref{l4.83}.
\end{proof}

We now introduce another important class of summation inequalities, 
whose proof follows analogously to the argument in \cite[p.~829]{sw1992}.

\begin{lemma}\label{lem3.x6}
Let $\alpha\in(0,n)$, $g\in C_c^\infty(\mathbb{R}^n)$. 
Let $\mu$ be a nonnegative locally finite Borel measure on $\mathbb{R}^n$. 
Then for any $Q\in \mathcal{D}^V$,
\begin{equation*}
\sum_{Q'\subseteq Q} |Q'|^{\frac{\alpha}{n}} \int_{Q'} |g(y)| \, d\mu(y)
\lesssim 
|Q|^{\frac{\alpha}{n}} \int_{Q} |g(y)| \, d\mu(y).
\end{equation*}
\end{lemma}

To link the geometric structure of $\mathcal{D}^V$ with analytic tools 
for operator estimates, we introduce the adapted weight class $A_p^{\mathcal{D}^V}$.  

\begin{definition}
Let $V\in RH^{n/2}$ satisfy Assumptions $\mathbf{(A_1)}$--$\mathbf{(A_3)}$. 
Let $\mathcal{D}^V$ be the perturbed dyadic cube system given in Definition \ref{d3.1}.
For $1<p<\infty$, a weight $\omega$ belongs to the class $A_p^{\mathcal{D}^V}$ if
\begin{equation*}
\sup_{Q \in \mathcal{D}^V} 
\left( \frac{1}{|Q|} \int_{Q} \omega(x) \, dx \right)
\left( \frac{1}{|Q|} \int_{Q} \omega(x)^{-1/(p-1)} \, dx \right)^{p-1} < \infty.
\end{equation*}
For $p=1$, we say that $\omega \in A_{1}^{\mathcal{D}^V}$ if 
there exists a constant $C\in[1,\infty)$ such that 
for every cube $Q \in \mathcal{D}^V$,
\begin{equation*}
\frac{1}{|Q|}\int_{Q}{\omega(x)\,dx} \leqslant C \operatorname*{ess\,inf}_{x\in Q} \omega(x).
\end{equation*}
\end{definition}

The next lemma shows that the Gaussian function $e^{b|x|^2}$ for any $b\in\mathbb{R}$
belong to $A_1^{\mathcal{D}^V}$ with $V(x) = |x|^2$.

\begin{lemma}\label{lem3.x7}
Let $p\in [1, \infty)$ and let $V(x)=|x|^2$ be the quadratic potential. 
For any $b \in \mathbb{R}$, the Gaussian weight
\begin{equation*}
\omega_0(x) := e^{b|x|^2}
\end{equation*}
belongs to $A_p^{\mathcal{D}^V}$.
\end{lemma}

\begin{proof}   
For any $Q\in \mathcal{D}^V$, let $x_Q$ be the central of $Q$. 
By (iii) of Theorem \ref{t1.2} and \eqref{eqn-ECR}, 
we have $l(Q)\lesssim \rho(x,V)\sim \min\{1,|x|^{-1}\}$.
Then for any $x\in Q$, by \cite[Proposition 2.1]{mm2007}, we obtain
\begin{equation}\label{eqn-Gaus}
\omega_0(x) \sim e^{b|x_Q|^2}.
\end{equation}

When $p\in [1, \infty)$, it follows from \eqref{eqn-Gaus} that
\begin{align*}
\left( \frac{1}{|Q|} \int_{Q} \omega_0(x) \, dx \right)
\left( \frac{1}{|Q|} \int_{Q} \omega_0(x)^{1-p'} \, dx \right)^{p-1} 
&\sim \left( \frac{1}{|Q|} \int_{Q} e^{b|x_Q|^2} \, dx \right)
\left( \frac{1}{|Q|} \int_{Q} e^{(1-p')b|x_Q|^2} \, dx \right)^{p-1} \\
&\sim e^{b|x_Q|^2} e^{-b|x_Q|^2} \sim 1,
\end{align*}
which shows that $A_p^{\mathcal{D}^V}$.

For $p=1$, from \eqref{eqn-Gaus}, we deduce that
\begin{equation*}
\frac{1}{|Q|}\int_{Q}{\omega_{0}(x)\,dx}
\left[\underset{x\in Q}{\mathrm{ess}\inf}\omega_{0}(x)\right]^{-1}
\sim\frac{1}{|Q|}\int_{Q}{\mathrm{e}^{b|c_{Q}|^{2}} dx} 
\cdot\mathrm{e}^{-b|c_{B}|^{2}}\sim 1,
\end{equation*}
which shows that $\omega \in A_{1}^{\mathcal{D}^V}$.
\end{proof}

\subsection{Adjacent dyadic systems}\label{s3.3}

In this subsection, we construct a family of adjacent perturbed dyadic cubes 
$\{\mathcal{D}^{V,t}\}_{t}$ so that the adjacent property holds true.
We begin by constructing the adjacent dyadic cube systems on the real line $\mathbb{R}$. 
To be precise, for any $k,l\in,\mathbb{Z}_+$, 
let $a_l$ and $b_l$ be dyadic scales and layer scales 
defined in \eqref{eqn-LP} and \eqref{eqn-LSS}, respectively. 
With $a_{-1}:=0$, we introduce the quantity
\begin{equation}\label{eqn-quan}
\overline{M}_{k,l}:= \frac{a_{l}-a_{l-1}}{4^{-k}b_l}.
\end{equation}
Then define the dyadic cube families on $\mathbb{R}$ by setting
\begin{equation*}
\mathcal{D}_{k,l}^{V,0}(\mathbb{R}) 
:= \left\{
    \begin{aligned}
    &\big[ a_{l-1} + m4^{-k}b_l,\, a_{l-1} + (m+1)4^{-k}b_l \big), \\
    &\big[ -a_l + m4^{-k}b_l,\, -a_l + (m+1)4^{-k}b_l \big)
    \end{aligned}
    :\ m = 0, 1, \dots, \overline{M}_{k,l} - 1
\right\}
\end{equation*}
and
\begin{equation*}
\mathcal{D}_{k}^{V,0}(\mathbb{R}):=
\bigcup_{l\in \mathbb{Z}_+} \mathcal{D}_{k,l}^{V,0}(\mathbb{R}), \quad
\mathcal{D}^{V,0}(\mathbb{R}):=
\bigcup_{k\in \mathbb{Z}_+} \mathcal{D}_{k}^{V,0}(\mathbb{R}).
\end{equation*}
Note that $\mathcal{D}^{V,0}(\mathbb{R})$ is nothing but 
a refinement of the previous family of perturbed dyadic cubes 
$\mathcal{D}^{V}(\mathbb{R})$ constructed in Theorem \ref{t1.2}.   

To define a shifted variant $\mathcal{D}^{V,1/3}(\mathbb{R})$ 
of $\mathcal{D}^{V,0}(\mathbb{R})$, 
for any $k,l\in \mathbb{Z}_+$, let
\begin{equation*}
\overline{\mathcal{D}}_{k,l}^{V,1/3}(\mathbb{R}) := 
\left\{ 
\begin{aligned}
&\left[ a_{l-1}+\left( m + \frac{1}{3} \right)4^{-k}b_l,\,a_{l-1}+\left( m + \frac{4}{3} \right)4^{-k}b_l \right), \\
&\left[ -a_l + \left( m + \frac{1}{3} \right)4^{-k}b_l,\, -a_l + \left( m + \frac{4}{3} \right)4^{-k}b_l \right)
\end{aligned} 
:\, 
m = 0, 1, \dots, \overline{M}_{k,l} - 2
\right\}.
\end{equation*}
For arbitrary $k\in \mathbb{Z}_+$ and $l\in \mathbb{N}$, 
we further define
\begin{equation*}
\widetilde{\mathcal{D}}_{k,0}^{V,1/3}(\mathbb{R}):=
\left\{
\left[ a_0-\frac{2}{3}4^{-k}b_0,\, a_0+\frac{1}{3}4^{-k}b_1 \right)
\right\},
\end{equation*}

\begin{equation*}
\widetilde{\mathcal{D}}_{k,l}^{V,1/3}(\mathbb{R}):=
\left\{
\left[ a_l-\frac{2}{3}4^{-k}b_l,\, a_l+\frac{1}{3}4^{-k}b_{l+1} \right),\quad
\left[ -a_{l-1}-\frac{2}{3}4^{-k}b_l,\, -a_{l-1}+\frac{1}{3}4^{-k}b_{l-1} \right)
\right\},
\end{equation*}
and
\begin{equation*}
\mathcal{D}_{k,l}^{V,1/3}(\mathbb{R}):=
\overline{\mathcal{D}}_{k,l}^{V,1/3}(\mathbb{R}) \cup 
\widetilde{\mathcal{D}}_{k,l}^{V,1/3}(\mathbb{R}).
\end{equation*}
The shifted dyadic family is defined by setting
\begin{equation*}
\mathcal{D}_{k}^{V,1/3}(\mathbb{R}):=
\bigcup_{l\in \mathbb{Z}_+} \mathcal{D}_{k,l}^{V,1/3}(\mathbb{R}) \quad
\text{and}\quad \mathcal{D}^{V,1/3}(\mathbb{R}):=
\bigcup_{k\in \mathbb{Z}_+} \mathcal{D}_{k}^{V,1/3}(\mathbb{R}).
\end{equation*}

Using the tensor product, we now extend the above construction from one-dimension to $\mathbb{R}^n$.

\begin{definition}\label{d3.3}
Let $V\in RH^{n/2}$ satisfy Assumptions $\mathbf{(A_1)}$–$\mathbf{(A_3)}$.
For any $t\in \{0, 1/3\}^n$, the adjacent dyadic cube systems 
$\mathcal{D}^{V,t}(\mathbb{R}^n)$ is defined by setting
\begin{equation*}
\mathcal{D}_k^{V,t}(\mathbb{R}^n):= 
\mathcal{D}_k^{V,t_1}(\mathbb{R}) \times \cdots \times \mathcal{D}_k^{V,t_n}(\mathbb{R}) 
\quad \text{and} \quad 
\mathcal{D}^{V,t}(\mathbb{R}^n):=
\bigcup_{k\in \mathbb{Z}_+} \mathcal{D}_{k}^{V,t}(\mathbb{R}^n).
\end{equation*}
\end{definition}

The adjacent dyadic cube systems $\mathcal{D}^{V,t}(\mathbb{R}^n)$ 
satisfy the following fundamental structural properties.

\begin{proposition}\label{p3.4}
Let $V\in RH^{n/2}$ satisfy Assumptions $\mathbf{(A_1)}$–$\mathbf{(A_3)}$.
For any $t\in \{0, 1/3\}^n$, the adjacent dyadic cube system 
$\mathcal{D}^{V,t}(\mathbb{R}^n)$ possesses the following key properties:
\begin{enumerate}
\item[\textnormal{(i)}] 
If $Q, P\in \mathcal{D}^{V,t}(\mathbb{R}^n)$, then $Q\cap P \in \{\emptyset, P, Q\}$; 
\item[\textnormal{(ii)}] 
For each $k\in \mathbb{N}$, the collection $\mathcal{D}_k^{V,t}(\mathbb{R}^n)$ 
constitutes a partition of $\mathbb{R}^n$.
\end{enumerate}
\end{proposition}

\begin{proof}
It suffices to prove Proposition \ref{p3.4} under the case $n=1$.
First, it is straightforward to verify that $\mathcal{D}^{V,0}(\mathbb{R})$ satisfies both (i) and (ii). 
By the definition of $\mathcal{D}^{V,1/3}(\mathbb{R})$, property (ii) holds trivially. Thus we only need to establish property (i) for $\mathcal{D}^{V,1/3}(\mathbb{R})$.

For any $k,l\in,\mathbb{Z}_+$, let $a_l$ and $b_l$ be dyadic scales and layer scales 
defined in \eqref{eqn-LP} and \eqref{eqn-LSS}, respectively.
Let $a_{-1}:=0$ and $\overline{M}_{k,l}$ be given by \eqref{eqn-quan}. 
We then define the sets of lattice points as follows:
\begin{equation*}
A_{k,l}^{1/3} :=
\left\{
\begin{aligned}
&a_{l-1}+\left(m+\frac{1}{3}\right)4^{-k}b_l, \\
&-a_l+\left(m+\frac{1}{3}\right)4^{-k}b_l
\end{aligned}
:\,
m=0,1,\dots,\overline{M}_{k,l}-1
\right\}.
\end{equation*}
For each $k\in\mathbb{Z}_+$, the lattice set associated with 
$\mathcal{D}_k^{V,1/3}(\mathbb{R})$ is defined by
\begin{equation}\label{e1-p1.1}
A_k^{1/3} := \bigcup_{l\in\mathbb{Z}_+} A_{k,l}^{1/3}.
\end{equation}

We now prove the nesting property
$
A_k^{1/3} \subseteq A_{k+1}^{1/3}, \quad \forall\,k\in\mathbb{Z}_+ .
$
Take any $a\in A_k^{1/3}$. Then $a\in A_{k,l}^{1/3}$ for some 
$l\in\mathbb{Z}_+$, we may assume without loss of generality that
\begin{equation*}
a = a_{l-1}+\left(m+\frac{1}{3}\right)4^{-k}b_l,\quad m=0,1,\dots,\overline{M}_{k,l}-1.
\end{equation*}
Then
\begin{equation*}
a = a_{l-1} + m4^{-k}b_l + 4^{-(k+1)}b_l + \frac{1}{3}4^{-(k+1)}b_l
= a_{l-1} + \left(4m+1+\frac{1}{3}\right)4^{-(k+1)}b_l.
\end{equation*}
Since
$
1\leq 4m+1\leq 4\overline{M}_{k,l}-3 = \overline{M}_{k+1,l}-3,
$
we have $a\in A_{k+1}^{1/3}$, and hence $A_{k,l}^{1/3}\subseteq A_{k+1}^{1/3}$.

The inclusion $A_k^{1/3}\subseteq A_{k+1}^{1/3}$ is therefore valid, which implies that $\mathcal{D}^{V,1/3}(\mathbb{R})$ has property (i)
and hence finishes the proof of Proposition \ref{p3.4}.
\end{proof}

Finally, we are in a position to show the adjacency property of the family of 
adjacent perturbed dyadic cubes $\{\mathcal{D}^{V,t}\}_{t}$.

\begin{proposition}\label{p3.5}
Let $V\in RH^{n/2}$ satisfy Assumptions
$\mathbf{(A_1)}$--$\mathbf{(A_3)}$.
Let $N_1\in\mathbb{N}$ be a sufficiently large constant.
For any rectangle $R=I_1\times\cdots\times I_n$ 
centered at $x=(x_1,\dots,x_n)$ and satisfying
$
l(I_i)\leq 4^{-N_1}\rho(\vec{x}_i,V),
$
$
\vec{x}_i:=(0,\dots,x_i,\dots,0),
$
there exist $t\in\{0,1/3\}^n$ and 
$R_t\in\mathcal{D}^{V,t}(\rn)$ such that
$$
R\subseteq R_t
\quad\text{and}\quad
|R_t|\lesssim|R|.
$$
\end{proposition}

\begin{proof}
Define
$
\widetilde{b}_0:=\min\{b_0,b_1\}$ and 
$
\widetilde{b}_l:=\min\{b_{l-1},b_1,b_{l+1}\}$
{for any } $l\in\mathbb{N}$.
Let $N_1\in\mathbb{N}$ be sufficiently large.
For any interval $I=I(x_I,l(I))$ with $l(I)\leq 4^{-N_1}\rho(x_I,V)$, 
there exists $l\in\mathbb{N}$ such that $x_I\in L_l$. 
By Lemmas~\ref{l2.7} and~\ref{l2.8},
$$
3l(I)\lesssim 3\cdot 4^{-N_1}\rho(x_I,V)
\sim 3\cdot 4^{-N_1}\widetilde{b}_l
<\widetilde{b}_l.
$$
We may therefore fix $k\in\mathbb{N}$ such that
$
4^{-(k+1)}\widetilde{b}_l
\leq 3l(I)
<4^{-k}\widetilde{b}_l.
$

We first show that for any such interval $I=I(x_I,l(I))$, 
there exist $t\in\{0,1/3\}$ and 
$I_t\in\mathcal{D}_{k}^{V,t}(\mathbb{R})$ 
satisfying
\begin{equation}\label{e1-t1.2}
I\subseteq I_t
\quad\text{and}\quad
l(I_t)\lesssim l(I).
\end{equation}

Suppose first that $x_I\in L_0$, and define
$
A_{k,0}:=\{m4^{-k}b_0,\ -a_0+m4^{-k}b_0
: m=0,1,\dots,M_{k,0}\}.
$
If $I$ contains no points from $A_{k,0}$, then there exists 
$I^0_1\in\mathcal{D}_{k,0}^V(\mathbb{R})$ such that
$I\subseteq I^0_1$ and $l(I^0_1)=4^{-k}b_0\lesssim l(I)$.
If $I$ intersects $A_{k,0}$, we distinguish two cases.

{\bf Case 1:} if $I$ contains $-a_0$, then there exists
$I^0_2\in\mathcal{D}_{k,1}^{V,1/3}(\mathbb{R})$ of the form
$$
I^0_2=\left[-a_0-\frac{2}{3}4^{-k}b_1,\
-a_0+\frac{1}{3}4^{-k}b_0\right)
$$
such that $I\subseteq I^0_2$ and $l(I^0_2)\sim 4^{-k}b_0\lesssim l(I)$.

{\bf Case 2:} if $I$ contains some point in $A_{k,0}\setminus\{-a_0\}$, 
then there exists $I^0_3\in\mathcal{D}_{k,0}^{V,1/3}(\mathbb{R})$ 
such that $I\subseteq I^0_3$ and $l(I^0_3)\sim 4^{-k}b_0\lesssim l(I)$.

Next, let $l\in\mathbb{N}$ and $x_I\in L_l$, and define
$
A_{k,l}:=\left\{a_{l-1}+m4^{-k}b_l,\ -a_l+m4^{-k}b_0
: m=0,1,\dots,\overline{M}_{k,l}\right\}.
$
If $I$ contains no points from $A_{k,l}$, 
then there exists $I^l_1\in\mathcal{D}_{k,l}^V(\mathbb{R})$ such that
$I\subseteq I^l_1$ and $l(I^l_1)=4^{-k}b_l\lesssim l(I)$.
If $I$ intersects $A_{k,l}$, we distinguish three cases.

{\bf Case i):} if $I$ contains $-a_l$, then there exists
$I^l_2\in\mathcal{D}_{k,l+1}^{V,1/3}(\mathbb{R})$ of the form
$$
I^l_2=\left[-a_l-\frac{2}{3}4^{-k}b_{l+1},\
-a_l+\frac{1}{3}4^{-k}b_l\right)
$$
such that
$I\subseteq I^l_2$ and $l(I^l_2)\sim 4^{-k}b_0\lesssim l(I)$.

{\bf Case ii):}
If $I$ contains $a_{l-1}$, then there exists
$I^l_3\in\mathcal{D}_{k,l-1}^{V,1/3}(\mathbb{R})$ of the form
$$
I^l_3=\left[a_{l-1}-\frac{2}{3}4^{-k}b_{l-1},\
a_{l-1}+\frac{1}{3}4^{-k}b_l\right)
$$
such that
$I\subseteq I^l_3$ and $l(I^l_3)\sim 4^{-k}b_l\lesssim l(I)$.

{\bf Case iii):}
If $I$ contains some point in $A_{k,l}\setminus\{-a_l,a_{l-1}\}$, 
then there exists
$I^l_4\in\mathcal{D}_{k,l}^{V,1/3}(\mathbb{R})$ such that
$I\subseteq I^l_4$ and $l(I^l_4)\sim 4^{-k}b_l\lesssim l(I)$.
This establishes \eqref{e1-t1.2}.

Now let $R=I_1\times\cdots\times I_n$ be any rectangle 
centered at $x=(x_1,\dots,x_n)$ with
$l(I_i)\leq 4^{-N_1}\rho(\vec{x}_i,V)$. 
By \eqref{e1-t1.2}, there exist $t=(t_1,\dots,t_n)\in\{0,1/3\}^n$ and
$$
R_t=I_{t_1}\times\cdots\times I_{t_n}
\in\mathcal{D}_k^{V,t_1}(\mathbb{R})\times\cdots\times\mathcal{D}_k^{V,t_n}(\mathbb{R})
=\mathcal{D}_k^{V,t}(\rn)
$$
such that
$$
R\subseteq R_t
\quad\text{and}\quad
|R_t|\lesssim|R|.
$$
This finishes the proof of Proposition \ref{p3.5}.
\end{proof}



\section{Perturbed sparse domination for Riesz potentials}\label{s4}

This section is devoted to the proofs of  Theorem \ref{p4.002} 
and Corollary \ref{t1.3}.
To this end, we establish a perturbed sparse domination for 
Riesz potential $(-\Delta+V)^{-\alpha/2}$
defined as in \eqref{eqn-fractionalp}.  

\subsection{Localization of Riesz potentials}\label{s4.1}

We begin with the following kernel estimate for $(-\Delta+V)^{-\alpha/2}$, 
which can be found in \cite[Proposition 3.3]{b2014}.

\begin{lemma}[\cite{b2014}]\label{l4.2}
Let $0<\alpha<n$, $V\in RH^{n/2}$ and 
$K_\alpha^V$ the integral kernel of the Riesz potential $(-\Delta+V)^{-\alpha/2}$. Then
for any $M>0$, there exists $C_M >0$ such that
for any $x,y\in \rn$,
\begin{equation*}
K_\alpha^V (x,y) \leq 
C_M \left(1+\frac{|x-y|}{\rho(x,V)}\right)^{-M} \frac{1}{|x-y|^{n-\alpha}}.
\end{equation*}
\end{lemma}

We next introduce a neighborhood structure adapted to 
the perturbed dyadic system $\mathcal{D}^V$.

\begin{definition}\label{d4.50}
Let $V\in RH^{n/2}$ satisfy Assumptions $\mathbf{(A_1)}$–$\mathbf{(A_3)}$.
Let $\mathcal{D}^V=\bigcup_{k\in\mathbb{Z}_+}\mathcal{D}_k^V$ 
be the perturbed dyadic cube system from Definition \ref{d3.1}.
For any $Q\in\mathcal{D}^V$, let $k\in\mathbb{N}$ be 
the unique index such that $Q\in\mathcal{D}_k^V$.
We define
\begin{equation*}
\mathcal{N}(Q):=
\bigl\{ \widetilde{Q}\in \mathcal{D}_k^V : d(\widetilde{Q},Q)=0 \bigr\}, 
\qquad
N(Q):= 
\bigcup_{\widetilde{Q}\in \mathcal{N}(Q)} \widetilde{Q}.
\end{equation*}
Here, $d(Q,\widetilde{Q})$ denotes the Euclidean distance 
between the cubes $Q$ and $\widetilde{Q}$, 
and $d(Q,\widetilde{Q})=0$ means the cubes are 
either coincident or share a common edge.
\end{definition}

Building on the neighborhood $N(Q)$, we define a localized version of 
the Riesz potential $(-\Delta+V)^{-\alpha/2}$.

\begin{definition}\label{d4.3}
Let $0<\alpha<n$ and $V\in RH^{n/2}$ satisfy Assumptions $\mathbf{(A_1)}$–$\mathbf{(A_3)}$. 
Suppose that $\mu$ is a nonnegative locally finite Borel measure on $\rn$. 
For any $g\in C_c^\infty(\rn)$ and $x\in \rn$, 
define the {\it localized Riesz potential} of $g$ by setting 
\begin{equation*}
I_{\alpha,\mathrm{loc}}^V g(x):= 
\int_{N(Q_x)}{\frac{|g(y)|}{|x-y|^{n-\alpha}}d\mu(y)},
\end{equation*}
where $Q_x\in \mathcal{D}_{0}^V$, 
and $N(Q_x)$ is defined as in Definition~\ref{d4.50}. 
\end{definition}

The following proposition shows that Riesz potential $(-\Delta+V)^{-\alpha/2}$ can been 
bounded by the localized Riesz potential 
$I_{\alpha,\mathrm{loc}}^V$ defined as in Definition \ref{d4.3}.

\begin{proposition}\label{l4.4}
Let $p' \in (1,\infty)$, $\alpha\in (0,n)$ and 
$V\in RH^{n/2}$ satisfy Assumptions $\mathbf{(A_1)}$–$\mathbf{(A_3)}$. 
Suppose $\mu$ is a nonnegative locally finite Borel measure on $\rn$. 
Then there exists a positive constant $C$ such that for any $g\in C_c^\infty(\rn)$,
\begin{equation*}
\left\|(-\Delta+V)^{-\frac{\alpha}{2}}(g \mu) \right\|_{L^{p'}(\rn)}\le C
\left\|  I_{\alpha,\mathrm{loc}}^V  g \right\|_{L^{p'}(\rn)}.
\end{equation*}
\end{proposition}

\begin{proof}   
Let $\mathcal{D}_0^V = \bigcup_{l\in \mathbb{Z}_+} \mathcal{D}_{0,l}^V$ 
denote the 0-th generation perturbed dyadic cube system 
given in Definition \ref{d3.1}. 
For each $x\in \mathbb{R}^n$, 
Theorem \ref{t1.2} ensures the existence of 
an integer $l_x \in \mathbb{Z}_+$ and 
a cube $Q_x \in \mathcal{D}_{0,l_x}^V$ 
such that $x \in Q_x$.
Let $\{L_l\}_{l \in \mathbb{Z}_+}$ denote 
the layers from Definition \ref{d3.1}. 
Let $N(Q_x)$ be the neighborhood given in Definition \ref{d4.3}. 
We now define
\begin{equation*}
F_1 (Q_x) :=
\begin{cases}
    \bigcup_{i=0,1} L_{l_x + i} \setminus N(Q_x),    & l_x = 0, \\
    \bigcup_{i=-1,0,1} L_{l_x + i} \setminus N(Q_x), & l_x \geq 1
\end{cases}
\end{equation*}
and 
\begin{equation*}
F_2 (Q_x) :=
\begin{cases}
    \bigcup_{i=2}^{\infty} L_{l_x + i}, & l_x = 0,1, \\
    \bigcup_{i=2}^{\infty} L_{l_x + i} \cup 
    \bigcup_{i=2}^{l_x} L_{l_x - i}, & l_x \geq 2.
\end{cases}
\end{equation*}
Note that
$\mathbb{R}^n = N(Q_x) \cup F_1(Q_x) \cup F_2(Q_x)$.
For any $M > 0$, Lemma \ref{l4.2} implies that
\begin{align*}
\left\|(-\Delta+V)^{-\frac{\alpha}{2}}(g \mu) \right\|_{L^{p'}(\rn)} 
&\leq 
\biggl[
\int_{\mathbb{R}^n} \biggl(
\int_{\mathbb{R}^n} \biggl|K_\alpha^V (x,y)g(y)\biggr| d\mu(y)
\biggr)^{p'} dx
\biggr]^{\frac{1}{p'}} \\
&\lesssim 
\biggl[ 
\int_{\mathbb{R}^n} \biggl(
\int_{\mathbb{R}^n} 
\biggl(1 + \frac{|x-y|}{\rho(x,V)}\biggr)^{-M} 
\frac{|g(y)|}{|x-y|^{n-\alpha}} d\mu(y)
\biggr)^{p'} dx 
\biggr]^{\frac{1}{p'}} \\
&\lesssim 
\biggl[ 
\int_{\mathbb{R}^n} \biggl(
\int_{N(Q_x)} 
\biggl(1 + \frac{|x-y|}{\rho(x,V)}\biggr)^{-M} 
\frac{|g(y)|}{|x-y|^{n-\alpha}} d\mu(y)
\biggr)^{p'} dx 
\biggr]^{\frac{1}{p'}}  \\
&\quad + 
\biggl[ 
\int_{\mathbb{R}^n} \biggl(
\int_{F_1(Q_x)} 
\cdots 
\biggr)^{p'} dx 
\biggr]^{\frac{1}{p'}} 
+ 
\biggl[ 
\int_{\mathbb{R}^n} \biggl(
\int_{F_2(Q_x)} 
\cdots 
\biggr)^{p'} dx 
\biggr]^{\frac{1}{p'}} \\
&=: I_1 + I_2 + I_3.
\end{align*}

We first estimate $I_1$. It is straightforward that
\begin{align*}
I_1 
&\leq \biggl[ \int_{\mathbb{R}^n} \biggl(
\int_{N(Q_x)} \frac{|g(y)|}{|x-y|^{n-\alpha}} \, d\mu(y)
\biggr)^{p'} dx \biggr]^{\frac{1}{p'}} 
\sim \biggl[ \int_{\mathbb{R}^n} 
\biggl( I_{\alpha,\mathrm{loc}}^V g(x) \biggr)
^{p'} dx \biggr]^{\frac{1}{p'}}.
\end{align*}

Next, we turn to the estimates of $I_2$. 
By the definition of $F_1(Q_x)$, we know that for all $y \in F_1(Q_x)$, 
there exists a constant $c_0 > 0$ such that $|y - x| \geq c_0 \rho(x,V)$.
For each $k \in \mathbb{N}$, let
\begin{equation}\label{e1.1-l4.4}
A_k := 
\left\{
y\in F_1(Q_x): c_0 \rho(x,V)k\leq |y - x|\leq c_0 \rho(x,V)(k+1)   
\right\}.
\end{equation}
By \eqref{e1.1-l4.4}, it follows that
\begin{align}\label{e1.2-l4.4}
I_2 &\sim
\biggl[ \int_{\mathbb{R}^n} \biggl(
\int_{F_1(Q_x)}
\biggl(1 + \frac{|y - x|}{\rho(x,V)}\biggr)^{-M} 
\frac{|g(y)|}{|y - x|^{n-\alpha}} \, d\mu(y)
\biggr)^{p'} dx \biggr]^{\frac{1}{p'}}\\ \notag
&\lesssim 
\biggl[ \int_{\mathbb{R}^n} \biggl(
\sum_{k=1}^{\infty}
\int_{A_k}
\biggl(1 + \frac{|y - x|}{\rho(x,V)}\biggr)^{-M} 
\frac{|g(y)|}{|y - x|^{n-\alpha}} \, d\mu(y)
\biggr)^{p'} dx \biggr]^{\frac{1}{p'}} \\ \notag
&\sim
\biggl[ \int_{\mathbb{R}^n} \biggl(
\sum_{k=1}^{\infty}
\int_{A_k}
k^{-M} \frac{|g(y)|}{\bigl(k\rho(x,V)\bigr)^{n-\alpha}} \, d\mu(y)
\biggr)^{p'} dx \biggr]^{\frac{1}{p'}} \\ \notag
&\lesssim 
\sum_{k=1}^{\infty} k^{-M+\alpha-n}
\biggl[ \int_{\mathbb{R}^n} \biggl(
\int_{A_k}
\frac{|g(y)|}{\rho(x,V)^{n-\alpha}} \, d\mu(y)
\biggr)^{p'} dx \biggr]^{\frac{1}{p'}}.
\end{align}
By definition of the neighborhood $N(Q_x)$, 
there exists a constant $c_1 >0$ such that
\begin{equation}\label{e1.3-l4.4}
Q(x, c_1 \rho(x,V)) \subseteq N(Q_x).
\end{equation}
For each fixed $k \in \mathbb{Z}^+$, the set $A_k$ can be covered by finitely many pairwise disjoint cubes of side length $c_1 \rho(x,V)$ by a standard Euclidean covering argument. More precisely,
\begin{equation}\label{e1.4-l4.4}
A_k \subseteq \bigcup_{j=1}^{N(k)} Q\big(x+z_j^k, c_1 \rho(x,V)\big),
\end{equation}
where $\{x+z_j^k\}_{j=1}^{N(k)}$ denote the centers of the covering cubes, and the covering number $N(k)$ satisfies the estimate
\begin{equation}\label{e1.5-l4.4}
N(k)
\sim \frac{|A_k|}{\big|Q\big(x+z_j^k, c_1 \rho(x,V)\big)\big|}
\sim \frac{\bigl(\rho(x,V)(k+1)\bigr)^n - \bigl(\rho(x,V)k\bigr)^n}{\rho(x,V)^n}
\lesssim k^n.
\end{equation}
By the definition of $A_k$ and 
\eqref{e1.4-l4.4}, \eqref{e1.5-l4.4} and \eqref{e1.3-l4.4}, 
we obtain
\begin{align*}
I_2 
&\lesssim 
\sum_{k=1}^{\infty} k^{-M+\alpha-n} \left[ \int_{\mathbb{R}^n} 
\left( \sum_{j=1}^{N(k)} 
\int_{Q_j\left(x+z_j^k, c_1 \rho(x,V)\right)} 
\frac{|g(y)|}{\rho(x,V)^{n-\alpha}} d\mu(y) \right)^{p'} dx 
\right]^{\frac{1}{p'}} \\
&\sim 
\sum_{k=1}^{\infty} k^{-M+\alpha-n} \left[ \int_{\mathbb{R}^n} 
\left( \sum_{j=1}^{N(k)} \int_{Q_j\left(x+z_j^k, c_1 \rho\left(x+z_j^k,V\right)\right)} \frac{|g(y)|}{\rho\left(x+z_j^k,V\right)^{n-\alpha}} 
d\mu(y) \right)^{p'} dx \right]^{\frac{1}{p'}} \\
&\sim
\sum_{k=1}^{\infty} k^{-M+\alpha-n} N(k)
\biggl[ \int_{\mathbb{R}^n} \biggl(
\int_{Q(x, c_1 \rho(x,V))}
\frac{|g(y)|}{\rho(x,V)^{n-\alpha}} \, d\mu(y)
\biggr)^{p'} dx \biggr]^{\frac{1}{p'}} \\
&\lesssim
\biggl[ \int_{\mathbb{R}^n} \biggl(
\int_{Q(x, c_1 \rho(x,V))}
\frac{|g(y)|}{|x-y|^{n-\alpha}} \, d\mu(y)
\biggr)^{p'} dx \biggr]^{\frac{1}{p'}} \\
&\lesssim
\biggl[ \int_{\mathbb{R}^n} \biggl(
\int_{N(Q_x)}
\frac{|g(y)|}{|x-y|^{n-\alpha}} \, d\mu(y)
\biggr)^{p'} dx \biggr]^{\frac{1}{p'}} \sim 
\biggl[ \int_{\mathbb{R}^n} 
\biggl( I_{\alpha,\mathrm{loc}}^V g(x) 
\biggr)^{p'} dx \biggr]^{\frac{1}{p'}}.
\end{align*}

We now turn to the estimate of $I_3$. 
Without loss of generality, we assume $l_x \geq 2$. This then implies
\begin{align*}
I_3 &=
\biggl[ \int_{\mathbb{R}^n} \biggl(
\int_{F_2(Q_x)}
\biggl(1 + \frac{|x - y|}{\rho(x,V)}\biggr)^{-M} 
\frac{|g(y)|}{|x - y|^{n-\alpha}} \, d\mu(y)
\biggr)^{p'} dx \biggr]^{\frac{1}{p'}}\\
&=
\biggl[ \int_{\mathbb{R}^n} \biggl(
\sum_{i=2}^\infty \int_{L_{l_x + i}}
\biggl(1 + \frac{|x - y|}{\rho(x,V)}\biggr)^{-M} 
\frac{|g(y)|}{|x - y|^{n-\alpha}} \, d\mu(y)
\biggr)^{p'} dx \biggr]^{\frac{1}{p'}}\\
&\quad +
\biggl[ \int_{\mathbb{R}^n} \biggl(
\sum_{i=2}^{l_x} \int_{L_{l_x - i}}
\biggl(1 + \frac{|x - y|}{\rho(x,V)}\biggr)^{-M} 
\frac{|g(y)|}{|x - y|^{n-\alpha}} \, d\mu(y)
\biggr)^{p'} dx \biggr]^{\frac{1}{p'}} \\
&=: I_{3,1} + I_{3,2}.
\end{align*}

To estimate $I_{3,1}$, for each $i \geq 2$ and $y \in L_{l_x + i}$, 
(i) of Lemma \ref{l3.11} yields a constant $C_G \in (0,1)$ such that
$|y - x| \gtrsim 2^{(1-C_G)l_x} 2^{i}\rho(x,V)$,
from which, it follows that
\begin{align}\label{e1.8-l4.4}
I_{3,1} 
&=
\biggl[ \int_{\mathbb{R}^n} \biggl(
\sum_{i=2}^\infty \int_{L_{l_x + i}}
\biggl(1 + \frac{|x - y|}{\rho(x,V)}\biggr)^{-M} 
\frac{|g(y)|}{|x - y|^{n-\alpha}} \, d\mu(y)
\biggr)^{p'} dx \biggr]^{\frac{1}{p'}} \\ \notag
&\lesssim
\biggl[ \int_{\mathbb{R}^n} \biggl(
\sum_{i=2}^\infty \int_{L_{l_x + i}}
2^{-(1-C_G)Ml_x} 2^{-Mi}
\frac{|g(y)|}{\bigl(2^{(1-C_G)l_x} 2^{i} \rho(x,V)\bigr)^{n-\alpha}} \, d\mu(y)
\biggr)^{p'} dx \biggr]^{\frac{1}{p'}} \\ \notag
&\sim\!
\biggl[ \int_{\mathbb{R}^n} \biggl(
\sum_{i=2}^\infty \int_{L_{l_x + i}}
2^{-(1-C_G)(M+n-\alpha)l_x} 2^{-(M+n-\alpha)i}
\frac{|g(y)|}{(\rho(x,V))^{n-\alpha}} \, d\mu(y)
\biggr)^{p'} dx \biggr]^{\frac{1}{p'}} .
\end{align}

By Lemma \ref{l2.8}, there exists a constant $k_{RD} \in \mathbb{Z}$ such that for all 
$i \geq 2$ and $y \in L_{l_x + i}$,
\begin{equation}\label{e1.9-l4.4}
\rho(y,V) \leq 2^{-ik_{RD}} \rho(x,V).
\end{equation}
Let $M$ be chosen sufficiently large. 
Using \eqref{e1.8-l4.4} and \eqref{e1.9-l4.4}, we deduce that
\begin{align}\label{e1.10-l4.4}
I_{3,1} 
&\lesssim
\biggl[ \int_{\mathbb{R}^n} \biggl(
\sum_{i=2}^\infty \int_{L_{l_x + i}}
2^{-(1-C_G)(M+n-\alpha)l_x} 2^{-(M+n-\alpha)i}
\frac{|g(y)|}{\bigl(2^{ik_{RD}}\rho(y,V)\bigr)^{n-\alpha}} \, d\mu(y)
\biggr)^{p'} dx \biggr]^{\frac{1}{p'}} \\ \notag
&\sim\!
\biggl[ \int_{\mathbb{R}^n} \biggl(
\sum_{i=2}^\infty 2^{-(1-C_G)(M+n-\alpha)l_x} 2^{-[M+(1+k_{RD})(n-\alpha)]i}
\int_{L_{l_x + i}} \frac{|g(y)|}{(\rho(y,V))^{n-\alpha}} \, d\mu(y)
\biggr)^{p'} dx \biggr]^{\frac{1}{p'}} \\ \notag
&\sim\!
\biggl[ \int_{\mathbb{R}^n} \biggl(
\sum_{i=2}^\infty 2^{-(1-C_G)M l_x} 2^{-Mi}
\int_{L_{l_x + i}} \frac{|g(y)|}{(\rho(y,V))^{n-\alpha}} \, d\mu(y)
\biggr)^{p'} dx \biggr]^{\frac{1}{p'}} \\ \notag
&\lesssim
\biggl[ \int_{\mathbb{R}^n} \biggl(
\sum_{i=2}^\infty 2^{-\frac{(1-C_G)M l_x}{p'}} 2^{-Mi}
\int_{L_{l_x + i}} \frac{|g(y)|}{(\rho(y,V))^{n-\alpha}} \, d\mu(y)
\biggr)^{p'} dx \biggr]^{\frac{1}{p'}}.
\end{align}
By \eqref{e1.10-l4.4} and H\"{o}lder inequality, we have
\begin{align*} 
(I_{3,1})^{p'}
&\lesssim \int_{\mathbb{R}^n} \biggl(
\sum_{i=2}^\infty 2^{-(\frac{1}{p}+\frac{1}{p'})Mi} 2^{-\frac{(1-C_G)M l_x}{p'}} 
\int_{L_{l_x +i}} \frac{|g(y)|}{\bigl(\rho(y,V)\bigr)^{n-\alpha}} \, d\mu(y)
\biggr)^{p'} dx \\
&\lesssim \sum_{i=2}^\infty 2^{-Mi}
\int_{\mathbb{R}^n} 2^{-(1-C_G)M l_x} \biggl( \int_{L_{l_x +i}} \frac{|g(y)|}{\bigl(\rho(y,V)\bigr)^{n-\alpha}} \, d\mu(y) \biggr)^{p'} dx\\
&\lesssim\! \sum_{i=2}^\infty 2^{-Mi} \sum_{j=0}^\infty 
\int_{L_j} 2^{-(1-C_G)M j} \biggl( \int_{L_{j +i}} \frac{|g(y)|}{\bigl(\rho(y,V)\bigr)^{n-\alpha}} \, d\mu(y) \biggr)^{p'} dx \\
&\sim\! \sum_{i=2}^\infty 2^{-Mi} \sum_{j=0}^\infty 2^{-(1-C_G)M j} \biggl( \int_{L_{j +i}} \frac{|g(y)|}
{\bigl(\rho(y,V)\bigr)^{n-\alpha}} \, d\mu(y) \biggr)^{p'} |L_j| \\
&\sim \sum_{i=2}^\infty 2^{-Mi} \sum_{j=0}^\infty 2^{-(1-C_G)M j} \biggl( \int_{L_{j +i}} \frac{|g(y)|}
{\bigl(\rho(y,V)\bigr)^{n-\alpha}} \, d\mu(y) \biggr)^{p'} N(j) b_j^n, 
\end{align*}
By this and Lemmas \ref{l3.12} and \ref{l2.7}, 
we further deduce that
\begin{align}\label{e1.140-l4.4}
(I_{3,1})^{p'}
&\lesssim\! \sum_{i=2}^\infty 2^{-Mi} \sum_{\widetilde{j}=i}^\infty 2^{-(1-C_G)M (\widetilde{j}-i)}
N(\widetilde{j}-i) b_{\widetilde{j}-i}^n \biggl( \int_{L_{\widetilde{j}}} \frac{|g(y)|}{\rho(y,V)^{n-\alpha}} \, d\mu(y) \biggr)^{p'} \\ \notag
&\sim\! \sum_{i=2}^\infty 2^{-M_1 i} \sum_{\widetilde{j}=i}^\infty 2^{-M_2 \widetilde{j}}
N(\widetilde{j}-i) b_{\widetilde{j}-i}^n \biggl( \int_{L_{\widetilde{j}}} \frac{|g(y)|}{\rho(y,V)^{n-\alpha}} \, d\mu(y) \biggr)^{p'}\\ \notag
&\lesssim \sum_{i=2}^\infty 2^{-M_1 i} \sum_{\widetilde{j}=i}^\infty 2^{-M_2 \widetilde{j}}
\bigl(2^{-(1-k_D)ni} 2^{(1-k_D)n\widetilde{j}} \bigr) \bigl(2^{-\widetilde{k_D} ni} b_{\widetilde{j}}^n\bigr)
\biggl( \int_{L_{\widetilde{j}}} \frac{|g(y)|}{\bigl(\rho(y,V)\bigr)^{n-\alpha}} \, d\mu(y) \biggr)^{p'} \\ \notag
&\lesssim \sum_{i=2}^\infty 2^{-\widetilde{M}_1 i} \sum_{\widetilde{j}=i}^\infty 2^{-\widetilde{M}_2 \widetilde{j}} b_{\widetilde{j}}^n \biggl( \int_{L_{\widetilde{j}}} \frac{|g(y)|}{\bigl(\rho(y,V)\bigr)^{n-\alpha}} \, d\mu(y) \biggr)^{p'} \\ \notag
&\lesssim \sum_{\widetilde{j}=i}^\infty 2^{-\widetilde{M}_2 \widetilde{j}} b_{\widetilde{j}}^n 
\biggl( \int_{L_{\widetilde{j}}} \frac{|g(y)|}{\rho(y,V)^{n-\alpha}} \, d\mu(y) \biggr)^{p'} ,
\end{align}
where $M_1=C_G M$, $M_2 = (1-C_G)M$, 
$\widetilde{M}_1 = M_1 + n i$ and 
$\widetilde{M}_2 = M_2 - (1-k_D)n$. 
Moreover, using \eqref{e1.140-l4.4} and Lemma \ref{l3.12}, 
we further have
\begin{align}\label{e1.14-l4.4}
(I_{3,1})^{p'}
&\lesssim \sum_{\widetilde{j}=i}^\infty 2^{-\widetilde{M}_2 \widetilde{j}} b_{\widetilde{j}}^n 
\biggl( \int_{L_{\widetilde{j}}} \frac{|g(y)|}{\rho(y,V)^{n-\alpha}} \, d\mu(y) \biggr)^{p'} \\ \notag
&\lesssim \sum_{\widetilde{j}=i}^\infty 2^{-\widetilde{M}_2 \widetilde{j}} N\bigl(\widetilde{j}\bigr)^{p' -1} 
\sum_{l=1}^{N\bigl(\widetilde{j}\bigr)} |Q_{\widetilde{j},l}|
\biggl( \int_{Q_{\widetilde{j},l}} \frac{|g(y)|}{\rho(y,V)^{n-\alpha}} \, d\mu(y) \biggr)^{p'} \\ \notag
&\lesssim \sum_{\widetilde{j}=i}^\infty \sum_{l=1}^{N\bigl(\widetilde{j}\bigr)} \int_{Q_{\widetilde{j},l}} 
\biggl( \int_{Q_{\widetilde{j},l}} \frac{|g(y)|}{\rho(y,V)^{n-\alpha}} \, d\mu(y) \biggr)^{p'} dx\\ \notag
&\lesssim \sum_{\widetilde{j}=i}^\infty \sum_{l=1}^{N\bigl(\widetilde{j}\bigr)} \int_{Q_{\widetilde{j},l}} 
\biggl( \int_{Q_{\widetilde{j},l}} \frac{|g(y)|}{\rho(x,V)^{n-\alpha}} \, d\mu(y) \biggr)^{p'} dx \\ \notag
&\lesssim \sum_{\widetilde{j}=i}^\infty \sum_{l=1}^{N\bigl(\widetilde{j}\bigr)} \int_{Q_{\widetilde{j},l}} 
\biggl( \int_{Q_{\widetilde{j},l}} \frac{|g(y)|}{|x-y|^{n-\alpha}} \, d\mu(y) \biggr)^{p'} dx \\ \notag
&\lesssim \int_{\mathbb{R}^n} \biggl( \int_{N(Q_x)} \frac{|g(y)|}{|x-y|^{n-\alpha}} \, d\mu(y) \biggr)^{p'} dx 
\sim \int_{\mathbb{R}^n} \biggl( I_{\alpha,\mathrm{loc}}^V g(x) \biggr)^{p'} dx.  
\end{align}

We now turn to the estimation of $I_{3,2}$. For each $2 \leq i \leq l_x$ and $y \in L_{l_x - i}$, 
part (ii) of Lemma \ref{l3.11} yields a constant $C_G \in (0,1)$ for which
\begin{equation}\label{e1.15-l4.4}
|y - x| \gtrsim 2^{(1-C_G)l_x}\rho(x,V).
\end{equation}
From \eqref{e1.15-l4.4}, it follows that
\begin{align}\label{e1.16-l4.4}
I_{3,2} 
&=
\biggl[ \int_{\mathbb{R}^n} \biggl(
\sum_{i=2}^{l_x} \int_{L_{l_x -i}}
\biggl(1 + \frac{|x-y|}{\rho(x,V)}\biggr)^{-M} 
\frac{|g(y)|}{|x-y|^{n-\alpha}} \, d\mu(y)
\biggr)^{p'} dx \biggr]^{\frac{1}{p'}} \\ \notag 
&\lesssim
\biggl[ \int_{\mathbb{R}^n} \biggl(
\sum_{i=2}^{l_x} \int_{L_{l_x -i}}
2^{-(1-C_G)Ml_x} 
\frac{|g(y)|}{\bigl(2^{(1-C_G)l_x}  \rho(x,V)\bigr)^{n-\alpha}} \, d\mu(y)
\biggr)^{p'} dx \biggr]^{\frac{1}{p'}} \\ \notag
&\sim\!
\biggl[ \int_{\mathbb{R}^n} \biggl(
\sum_{i=2}^{l_x} \int_{L_{l_x -i}}
2^{-(1-C_G)(M+n-\alpha)l_x} 
\frac{|g(y)|}{\bigl(\rho(x,V)\bigr)^{n-\alpha}} \, d\mu(y)
\biggr)^{p'} dx \biggr]^{\frac{1}{p'}}
\end{align}

By Lemma \ref{l2.7}, there exists a constant $\widehat{k_D} \in \mathbb{Z}$ 
such that for all $2 \leq i \leq l_x$ and $y \in L_{l_x - i}$, 
the critical radius satisfies
\begin{equation}\label{e1.17-l4.4}
\rho(y,V) \leq 2^{-i\widehat{k_D}} \rho(x,V).
\end{equation}
Let $M$ be chosen sufficiently large. Using \eqref{e1.16-l4.4} and \eqref{e1.17-l4.4}, we deduce that
\begin{align}\label{e1.18-l4.4}
I_{3,2} 
&\lesssim
\biggl[ \int_{\mathbb{R}^n} \biggl(
\sum_{i=2}^{l_x} \int_{L_{l_x -i}}
2^{-(1-C_G)(M+n-\alpha)l_x} 
\frac{|g(y)|}{\bigl(2^{i\widehat{k_D}}\rho(y,V)\bigr)^{n-\alpha}} \, d\mu(y)
\biggr)^{p'} dx \biggr]^{\frac{1}{p'}} \\ \notag
&\lesssim
\biggl[ \int_{\mathbb{R}^n} \biggl(
\sum_{i=2}^{l_x} \int_{L_{l_x -i}}
2^{-(1-C_G)(M+n-\alpha)l_x/2} 2^{-(1-C_G)(M+n-\alpha)i/2}
\frac{|g(y)|}{\bigl(2^{i\widehat{k_D}}\rho(y,V)\bigr)^{n-\alpha}} \, d\mu(y)
\biggr)^{p'} dx \biggr]^{\frac{1}{p'}} \\ \notag
&\lesssim
\biggl[ \int_{\mathbb{R}^n} \biggl(\sum_{i=2}^{l_x} 2^{-M_0 l_x} 
\int_{L_{l_x -i}} \frac{|g(y)|}{\bigl(\rho(y,V)\bigr)^{n-\alpha}} \, d\mu(y)
\biggr)^{p'} dx \biggr]^{\frac{1}{p'}},
\end{align}
where $M_0:= (1-C_G)(M+n-\alpha)/2$.
Recall from Definition~\ref{d3.1} that $\mathbb{R}^n$ is 
partitioned into the layers $\{L_j\}_{j\in\mathbb{Z}_+}$.
Utilizing this layer decomposition together with \eqref{e1.18-l4.4}, we obtain
\begin{align}\label{e1.19-l4.4}
(I_{3,2})^{p'}
&=\sum_{j=0}^{\infty} \int_{L_j} \biggl(
\sum_{i=2}^{j} 2^{-M_0 j} 
\int_{L_{j -i}} \frac{|g(y)|}{\bigl(\rho(y,V)\bigr)^{n-\alpha}} \, d\mu(y)
\biggr)^{p'} dx \\ \notag
&=\sum_{j=0}^{\infty} 2^{-p' M_0 j}  \biggl( \sum_{i=2}^{j}   
\int_{L_{j -i}} \frac{|g(y)|}{\bigl(\rho(y,V)\bigr)^{n-\alpha}} \, d\mu(y)
\biggr)^{p'} |L_j| \\ \notag
&\leq \sum_{j=0}^{\infty} 2^{-p' M_0 j} (j-1)^{p' -1}  \sum_{i=2}^{j} \biggl(   
\int_{L_{j -i}} \frac{|g(y)|}{\bigl(\rho(y,V)\bigr)^{n-\alpha}} \, d\mu(y)
\biggr)^{p'} |L_j| \\ \notag
&\lesssim \sum_{j=0}^{\infty} 2^{-p' M_0 j} 2^{p' j}  \sum_{i=2}^{j} \biggl(   
\int_{L_{j -i}} \frac{|g(y)|}{\bigl(\rho(y,V)\bigr)^{n-\alpha}} \, d\mu(y)
\biggr)^{p'} |L_j| \\ \notag
&\sim\!  \sum_{j=0}^{\infty} 2^{-\widetilde{M}_0 j}  \sum_{i=2}^{j} 
\biggl( \int_{L_{j -i}} \frac{|g(y)|}{\bigl(\rho(y,V)\bigr)^{n-\alpha}} \, d\mu(y) \biggr)^{p'} |L_j|, 
\end{align}
where we set $\widetilde{M}_0 = M_0 - 1$, and $M_0$ is chosen sufficiently large to 
guarantee that $M_1 > 0$. 
Based on \eqref{e1.19-l4.4} and an interchange the order of summation,  we find
\begin{align*} 
(I_{3,2})^{p'}
&\lesssim\!  \sum_{i=2}^{\infty} \sum_{j=i}^{\infty} 2^{-\widetilde{M}_0 j}
\biggl( \int_{L_{j -i}} \frac{|g(y)|}{\bigl(\rho(y,V)\bigr)^{n-\alpha}} \, d\mu(y) \biggr)^{p'} |L_j| \\
&\sim\!  \sum_{i=2}^{\infty} \sum_{j=i}^{\infty} 2^{-\widetilde{M}_0 j} \biggl( \int_{L_{j -i}} \frac{|g(y)|}{\bigl(\rho(y,V)\bigr)^{n-\alpha}} \, d\mu(y) \biggr)^{p'} N(j) b_j^n \\
&\sim\!  \sum_{i=2}^{\infty} \sum_{\widetilde{j}=0}^{\infty} 2^{-\widetilde{M}_0 (\widetilde{j}+i)} \biggl( \int_{L_{\widetilde{j}}} \frac{|g(y)|}{\bigl(\rho(y,V)\bigr)^{n-\alpha}} \, d\mu(y) \biggr)^{p'} 
N(\widetilde{j}+i) b_{\widetilde{j}+i}^n .
\end{align*}
By this and Lemmas \ref{l3.12}, \ref{l2.8},  we obtain
\begin{align}\label{e1.21-l4.4}
(I_{3,2})^{p'}
&\lesssim  \sum_{i=2}^{\infty} 2^{-\widetilde{M}_0 i}  \sum_{\widetilde{j}=0}^{\infty} 
2^{-\widetilde{M}_0 \widetilde{j}} \biggl( \int_{L_{\widetilde{j}}} 
\frac{|g(y)|}{\bigl(\rho(y,V)\bigr)^{n-\alpha}} \, d\mu(y) \biggr)^{p'} 2^{(1-k_D)n(\widetilde{j}+i)} 2^{-k_{RD} ni} b_{\widetilde{j}}^n \\ \notag
&\lesssim\sum_{\widetilde{j}=0}^{\infty} 2^{-\widehat{M}_0 \widetilde{j}} b_{\widetilde{j}}^n 
\biggl( \int_{L_{\widetilde{j}}} \frac{|g(y)|}{\bigl(\rho(y,V)\bigr)^{n-\alpha}} \, d\mu(y) \biggr)^{p'}, 
\end{align}
where $\widehat{M}_0=\widetilde{M}_0 - (1-k_D)n$.
Building on \eqref{e1.21-l4.4} and 
following exactly the same arguments as for \eqref{e1.14-l4.4}, 
we readily conclude that
\begin{equation*}
(I_{3,2})^{p'} \lesssim \int_{\mathbb{R}^n} 
\biggl( I_{\alpha,\mathrm{loc}}^V g(x) \biggr)^{p'} dx,
\end{equation*}
which completes the proof of Proposition \ref{l4.4}.
\end{proof}

\subsection{Perturbed sparse domination}\label{s4.2}

In this subsection, we establish a perturbed sparse domination for localized Riesz potential 
$I_{\alpha, \mathrm{loc}}^V$. We start with the following definition.

\begin{definition}\label{d4.5}
Let $\alpha\in(0,n)$ and $V\in RH^{n/2}$ satisfy Assumptions $\mathbf{(A_1)}$–$\mathbf{(A_3)}$.
Suppsoe $\mu$ is a nonnegative locally finite Borel measure on $\mathbb{R}^n$.
For any compactly supported smooth function $g\in C_c^\infty(\mathbb{R}^n)$ and $x\in\mathbb{R}^n$, 
the \emph{dyadic Riesz potential} of $g$ adapted to $\mathcal{D}^V$ is defined by
\begin{equation*}
I_{\alpha}^{\mathcal{D}^{V}} g(x):= \sum_{Q\in \mathcal{D}^{V}} \frac{1}{|Q|^{1-\frac{\alpha}{n}}} \left( \int_{Q} |g(y)| \, d\mu(y) \right) \mathbf{1}_{N(Q)}(x),
\end{equation*}
where $N(Q)$ is defined as in Definition~\ref{d4.50}.
\end{definition}

The next lemma establishes the pointwise domination of the localized Riesz 
potential by dyadic Riesz potential $I_{\alpha}^{\mathcal{D}^{V}}$ defined as
in Definition \ref{d4.5}.

\begin{lemma}\label{l4.6}
Let $\alpha\in(0,n)$ and $V\in RH^{n/2}$ satisfy Assumptions $\mathbf{(A_1)}$–$\mathbf{(A_3)}$.
Then for every $g\in C_c^\infty(\mathbb{R}^n)$ and $x\in\mathbb{R}^n$,
\begin{equation*}
I_{\alpha, \mathrm{loc}}^V g(x)\lesssim I_{\alpha}^{\mathcal{D}^{V}} g(x).
\end{equation*}
\end{lemma}

\begin{proof}   
Fix any $x\in\mathbb{R}^n$ and $k\in\mathbb{N}$. 
Let $Q_x^k\in\mathcal{D}_{k}^V$ be the unique cube containing $x$, 
and write $Q_x:=Q_x^0$ for brevity. 
By the nested property of the perturbed dyadic cubes, 
we have the decomposition
\begin{equation*}
N(Q_x)=\bigcup_{k=0}^\infty\Bigl(N(Q_x^k)\setminus N(Q_x^{k+1})\Bigr),
\end{equation*}
where $N(Q_x^k)$ is defined as in Definition~\ref{d4.50}.
with disjoint sets on the right-hand side. 
For all $y\in N(Q_x^k)\setminus N(Q_x^{k+1})$, 
we have the geometric lower bound $|x-y|\gtrsim\ell(Q_x^k)$. 
Combining this with Definition~\ref{d4.3}, we compute
\begin{align}\label{e1-l4.6}
I_{\alpha, \mathrm{loc}}^V g(x)
&=\int_{N(Q_x)} \frac{|g(y)|}{|x-y|^{n-\alpha}}\,d\mu(y) \\ \notag
&= \sum_{k=0}^\infty \int_{N(Q_x^k)\setminus N(Q_x^{k+1})} \frac{|g(y)|}{|x-y|^{n-\alpha}}\,d\mu(y) \\
\notag
&\lesssim \sum_{k=0}^\infty \ell(Q_x^k)^{\alpha-n} \int_{N(Q_x^k)} |g(y)|\,d\mu(y).
\end{align}
For a fixed scale $k$, recall that $\mathcal{D}_k^V$ forms 
a partition of $\mathbb{R}^n$, and $N(Q_x^k)$ is exactly the 
union of all cubes $Q\in\mathcal{D}_k^V$ such that $x\in N(Q)$. 
Since the integral over $N(Q_x^k)$ can be rewritten as a sum of integrals over disjoint cubes $Q\in\mathcal{D}_k^V$ with $x\in N(Q)$, then by \eqref{e1-l4.6} and Definition~\ref{d4.5}, we have
\begin{align*}
I_{\alpha, \mathrm{loc}}^V g(x)
&\lesssim \sum_{k=0}^\infty \sum_{Q\in\mathcal{D}_k^V} |Q|^{\frac{\alpha}{n}-1} \left(\int_Q |g(y)|\,d\mu(y)\right) \mathbf{1}_{N(Q)}(x) \\
&\sim \sum_{Q\in\mathcal{D}^V} |Q|^{\frac{\alpha}{n}-1} \left(\int_Q |g(y)|\,d\mu(y)\right) \mathbf{1}_{N(Q)}(x) 
\sim I_{\alpha}^{\mathcal{D}^{V}} g(x).
\end{align*}
This completes the proof of Lemma \ref{l4.6}.
\end{proof}

To establish the sparse domination of the dyadic Riesz potential, we introduce the notion of sparse 
family adapted to the perturbed dyadic system $\mathcal{D}^V$.

\begin{definition}\label{d4.001}
Let $V\in RH^{n/2}$ satisfy Assumptions $\mathbf{(A_1)}$–$\mathbf{(A_3)}$.
For any $\eta\in(0,1)$, a subcollection $\mathcal{S}^V\subseteq \mathcal{D}^{V}$ is called an \emph{$\eta$-sparse family} if for every $Q\in \mathcal{S}^V$, there exists a measurable subset 
$E_Q\subseteq Q$ such that $|E_Q|\geq \eta |Q|$ and the family $\{E_Q\}_{Q\in \mathcal{S}^V}$ 
is pairwise disjoint.
\end{definition}

\begin{remark}[\cite{h2018}]\label{r4.001}
$\mathcal{S}^V$ is $\eta$-sparse if and only if for any $Q\in \mathcal{S}^V$,
\begin{equation*}
\sum_{\substack{Q'\subseteq Q \\ Q'\in\mathcal{S}^V}} |Q'|
\leq \frac{1}{\eta} |Q|.
\end{equation*}
\end{remark}

With the notion of sparse families in hand, we now introduce the perturbed 
sparse Riesz potential as follows.

\begin{definition}\label{d4.7}
Let $\alpha\in(0,n)$, $V\in RH^{n/2}$ satisfy Assumptions $\mathbf{(A_1)}$–$\mathbf{(A_3)}$
and $\mathcal{S}^V\subseteq\mathcal{D}^V$ a perturbed 
sparse family. Suppose $\mu$ is a nonnegative locally finite Borel measure on $\mathbb{R}^n$.
For any $g\in C_c^\infty(\mathbb{R}^n)$ and $x\in\mathbb{R}^n$, the \emph{perturbed 
sparse Riesz potential} adapted to $\mathcal{S}^V$ is defined by
\begin{equation*}
I_{\alpha}^{\mathcal{S}^{V}} g(x):= \sum_{Q\in \mathcal{S}^{V}} 
\frac{1}{|Q|^{1-\frac{\alpha}{n}}} \left( \int_{Q} |g(y)| \, d\mu(y) \right) \mathbf{1}_{N(Q)}(x),
\end{equation*}
where $N(Q)$  is defined as in Definition~\ref{d4.50}.
\end{definition}

Next, we establish that the dyadic Riesz potential is pointwise dominated by its 
perturbed sparse operator $I_{\alpha}^{\mathcal{S}^{V}}$ defined as in 
Definition \ref{d4.7}.

\begin{proposition}\label{l4.8}
Let $\alpha\in(0,n)$ and $V\in RH^{n/2}$ satisfy Assumptions $\mathbf{(A_1)}$–$\mathbf{(A_3)}$.
Then for any $g\in C_c^\infty(\mathbb{R}^n)$, there exists a sparse family $\mathcal{S}^V\subseteq\mathcal{D}^V$ such that
\begin{equation*}
I_{\alpha}^{\mathcal{D}^{V}} g(x) \lesssim I_{\alpha}^{\mathcal{S}^{V}} g(x)
\end{equation*}
holds for all $x\in\mathbb{R}^n$, where the implicit constant is independent of both $g$ and $x$.
\end{proposition}

\begin{proof}   
Let $\mu$ be a nonnegative locally finite Borel measure on $\mathbb{R}^n$.
For each $k\in\mathbb{Z}$, define the level class
\begin{equation}\label{e1-l4.8}
\mathcal{Q}_{k}^V:=
\left\{P \in \mathcal{D}^{V}: 2^{k}<\frac{1}{|P|}\int_{P} |g(y)|\,d\mu(y) \leq 2^{k+1} \right\}.
\end{equation}
For any $x\in\mathbb{R}^n$, by Definition \ref{d4.5} and \eqref{e1-l4.8}, we have
\begin{equation}\label{e2-l4.6}
\begin{split}
I_{\alpha}^{\mathcal{D}^{V}} g(x)
&= \sum_{P \in \mathcal{D}^{V}} |P|^{\frac{\alpha}{n}-1} \left(\int_{P} |g(y)|\,d\mu(y)\right) \mathbf{1}_{N(P)}(x)
\\
&= \sum_{k \in \mathbb{Z}} \sum_{P\in \mathcal{Q}_{k}^V} 
   |P|^{\frac{\alpha}{n}-1} \left(\int_{P} |g(y)|\,d\mu(y)\right) \mathbf{1}_{N(P)}(x)
\leq \sum_{k \in \mathbb{Z}} 2^{k+1} \sum_{P\in \mathcal{Q}_{k}^V} 
|P|^{\frac{\alpha}{n}} \mathbf{1}_{N(P)}(x).
\end{split}
\end{equation}

Next, for any $k\in\mathbb{Z}$, let $\mathcal{S}_k^V$ be the family of disjoint, maximal cubes $Q\in\mathcal{D}^V$ such that
\begin{equation}\label{e3-l4.6}
\frac{1}{|Q|}\int_{Q} |g(y)|\,d\mu(y)>2^{k},
\end{equation}
and define the full sparse candidate family
\begin{equation}\label{e3.0-l4.6}
\mathcal{S}^V:=\bigcup_{k\in\mathbb{Z}}\mathcal{S}_{k}^V.
\end{equation}
By maximality, every $P\in\mathcal{Q}_k^V$ is contained in some $Q\in\mathcal{S}_k^V$.
Combined with \eqref{e2-l4.6}, this implies 
\begin{equation}\label{e4-l4.6}
I_{\alpha}^{\mathcal{D}^{V}} g(x)
\leq \sum_{k \in \mathbb{Z}} 2^{k+1}
\sum_{Q \in \mathcal{S}_{k}^V}
\sum_{\substack{P \subseteq Q \\ P \in \mathcal{D}^{V}}} 
|P|^{\frac{\alpha}{n}} \mathbf{1}_{N(P)}(x).
\end{equation}

We now establish the inner sum estimate:
\begin{equation}\label{e4.0-l4.6}
\sum_{\substack{P \subseteq Q \\ P \in \mathcal{D}^{V}}} 
|P|^{\frac{\alpha}{n}} \mathbf{1}_{N(P)}(x)
=\sum_{j=0}^{\infty} \sum_{\substack{P \subset Q,\, P \in \mathcal{D}^{V} \\ \ell(P)=2^{-j}\ell(Q)}}
 |P|^{\frac{\alpha}{n}}  \mathbf{1}_{N(P)}(x)
\lesssim \sum_{j=0}^{\infty} 2^{-j\alpha} |Q|^{\frac{\alpha}{n}} \mathbf{1}_{N(Q)}(x)
\lesssim |Q|^{\frac{\alpha}{n}} \mathbf{1}_{N(Q)}(x).
\end{equation}

From \eqref{e4-l4.6}, \eqref{e4.0-l4.6}, \eqref{e3-l4.6} and Definition \ref{d4.7}, 
we deduce that
\begin{equation*}
I_{\alpha}^{\mathcal{D}^{V}} g(x)
\lesssim
\sum_{k \in \mathbb{Z}} 2^{k+1}\sum_{Q \in \mathcal{S}_{k}^V} |Q|^{\frac{\alpha}{n}} \mathbf{1}_{N(Q)}(x)
\lesssim
\sum_{k\in\mathbb{Z}} \sum_{Q \in \mathcal{S}_{k}^V}
|Q|^{\frac{\alpha}{n}-1} \left(\int_{Q} |g(y)|\,d\mu(y)\right) \mathbf{1}_{N(Q)}(x) 
\sim I_{\alpha}^{\mathcal{S}^V} g(x).
\end{equation*}

It remains to verify that $\mathcal{S}^V$ is sparse. 
By Remark~\ref{r4.001}, it suffices to show that 
there exists a constant $C>1$ such that for any $Q\in \mathcal{S}^V$,
\begin{equation}\label{e5.0-l4.6}
\sum_{\substack{Q'\subseteq Q \\ Q'\in\mathcal{S}^V}} |Q'|
\leq C |Q|.
\end{equation}
Let $\mathrm{ch}_{\mathcal{S}^V}(Q)$ denote the collection of 
all maximal elements $Q'\in \mathcal{S}^V$ such that $Q'\subsetneq Q$. 
Then
\begin{equation}\label{e5-l4.6}
\sum_{\substack{Q'\subseteq Q \\ Q'\in\mathcal{S}^V}} |Q'|
=|Q| + \sum_{Q' \in \mathrm{ch}_{\mathcal{S}^V}(Q)} |Q'| 
     + \sum_{Q' \in \mathrm{ch}_{\mathcal{S}^V}(Q)}
        \sum_{\substack{Q''\subsetneq Q' \\ Q'' \in \mathcal{S}^V}} |Q''|.
\end{equation}
By \eqref{e3.0-l4.6}, there exists $k\in\mathbb{Z}$ such that 
$Q' \in \mathcal{S}_k^V$ for every $Q'\in \mathrm{ch}_{\mathcal{S}^V}(Q)$. 
Note that if $Q'' \in \mathcal{S}^V$ and $Q''\subsetneq Q'$, 
then $Q'' \in \mathcal{S}_{k+i}^V$ for some $i\in \{1,2,\ldots\}$. 
Therefore, applying \eqref{e3-l4.6}, we obtain
\begin{equation}\label{e6-l4.6}
\sum_{\substack{Q''\subsetneq Q' \\ Q'' \in \mathcal{S}^V}} |Q''|
= \sum_{i=1}^\infty 
   \sum_{\substack{Q''\subsetneq Q' \\ Q'' \in \mathcal{S}_{k+i}^V}} 
        |Q''|
<  \sum_{i=1}^\infty \frac{1}{2^{k+i}} 
        \sum_{\substack{Q''\subsetneq Q' \\ Q'' \in \mathcal{S}_{k+i}^V}}
          \int_{Q''} {|g(y)|d\mu(y)}
\leq  \sum_{i=1}^\infty \frac{1}{2^{k+i}} 
          \int_{Q'} {|g(y)|d\mu(y)}.        
\end{equation}
Since $Q' \in \mathcal{S}_k^V$ and $Q'\subsetneq Q$, 
we can choose $Q_0\in\mathcal{D}^V$ such that 
$Q_0\supseteq Q'$ and $|Q_0|=2^n|Q'|$. 
By the maximality of cubes in $\mathcal{S}_k^V$, 
it follows that
\begin{equation}\label{e7-l4.6}
\frac{1}{|Q'|}\int_{Q'} |g(y)|\,d\mu(y)\leq
\frac{2^n}{|Q_0|}\int_{Q_0} |g(y)|\,d\mu(y)\leq
2^n 2^k.
\end{equation}
Then combining \eqref{e5-l4.6}, \eqref{e6-l4.6} and \eqref{e7-l4.6}, we obtain
\begin{equation*}
\sum_{\substack{Q'\subseteq Q \\ Q'\in\mathcal{S}^V}} |Q'|
=|Q| + \sum_{Q' \in \mathrm{ch}_{\mathcal{S}^V}(Q)} |Q'| 
      +  \sum_{i=1}^\infty \frac{2^n 2^k}{2^{k+i}} 
          \sum_{Q' \in \mathrm{ch}_{\mathcal{S}^V}(Q)} |Q'|
\leq C|Q|,
\end{equation*}
which immediately yields \eqref{e5.0-l4.6}.
This finishes the proof of Proposition \ref{l4.8}.
\end{proof}

\subsection{Proof of Theorem \ref{p4.002}}\label{s4.3}

This subsection is devoted to the proof of Theorem \ref{p4.002}. 
We first recall the definition of the Carleson constant.
Let $p\in(1,\infty)$, $\alpha\in(0,n)$ and $V\in RH^{n/2}$ satisfy 
Assumptions $\mathbf{(A_1)}$--$\mathbf{(A_3)}$. Suppose $\mu$ is a nonnegative locally finite Borel measure on $\mathbb{R}^n$
and $\mathcal{S}^V\subseteq\mathcal{D}^V$ a sparse family.
The \emph{Carleson constant} adapted to $\mathcal{S}^V$ is defined by
\begin{align}\label{d4.002}
C_{p',\alpha}^V(\mu):=
\sup\left\{
\frac{1}{\mu(Q)}
\sum_{\substack{Q'\in \mathcal{S}^V\\ Q'\subseteq Q} }
|Q'|^{\frac{\alpha p'}{n}-(p'-1)} \mu(Q')^{p'}
\,:\ \, Q\in \mathcal{S}^V
\right\},
\end{align}
where $p'=p/(p-1)$.

The following lemma establishes key estimates for the Carleson constant 
$C_{p',\alpha}^V(\mu)$ adapted to $\mathcal{D}^V$.

\begin{lemma}\label{pp4.002}
Let $\alpha\in(0,n)$ and $V\in RH^{n/2}$ 
satisfy Assumptions $\mathbf{(A_1)}$--$\mathbf{(A_3)}$.
Suppose $\mu$ is a nonnegative locally finite Borel measure on $\mathbb{R}^n$.
Then the following statements hold true.
\begin{enumerate}
\item[\textnormal{(i)}]
If $1<p<\infty$, then
\begin{equation*}
C_{p',\alpha}^V(dx)
\sim \sup_{Q\in\mathcal{D}^V} |Q|^{\frac{\alpha p'}{n}}.
\end{equation*}

\item[\textnormal{(ii)}]
If $p>n/\alpha$, then
\begin{equation*}
C_{p',\alpha}^V(\mu) \sim
\sup_{Q\in\mathcal{D}^V} |Q|^{\frac{\alpha p'}{n}-(p'-1)} \mu(Q)^{p'-1}.
\end{equation*}
\end{enumerate}
\end{lemma}

To prove Lemma \ref{pp4.002}, we require the following lemma
from \cite[Lemma~4.2]{fh2018}, whose proof can been extended
to the current perturbed setting easily.

\begin{lemma}[\cite{fh2018}]\label{l4.101}
Let $V\in RH^{n/2}$ satisfy Assumptions $\mathbf{(A_1)}$--$\mathbf{(A_3)}$ and
$\mathcal{S}^V$ be a perturbed sparse family. 
Suppose $\delta>0$, $\gamma\geq 0$ and $\delta + \gamma \geq 1$. Then for any cube $Q$,
it holds
\begin{equation*}
\sum_{\substack{Q' \in \mathcal{S}^V \\ Q' \subset Q}} |Q'|^\delta \mu(Q')^\gamma 
\lesssim |Q|^\delta \mu(Q)^\gamma.
\end{equation*}
\end{lemma}

\begin{proof}[Proof of Lemma \ref{pp4.002}] 
We first prove (i).
For any $Q\in\mathcal{S}^V$, applying Lemma \ref{l4.101} with $\delta=\frac{\alpha p'}{n}+1$ 
and $\gamma=0$ gives
\begin{equation*}
\frac{1}{|Q|} \sum_{\substack{Q'\subseteq Q \\ Q'\in \mathcal{S}^V}}|Q'|^{\frac{\alpha p'}{n} + 1} 
\lesssim \frac{1}{|Q|} |Q|^{\frac{\alpha p'}{n} + 1}
 = |Q|^{\frac{\alpha p'}{n}} \lesssim \sup_{Q\in \mathcal{D}^V} |Q|^{\frac{\alpha p'}{n}}.
\end{equation*}
Taking the supremum over all $Q\in\mathcal{S}^V$ yields the upper bound
$C_{p',\alpha}^V(dx) \lesssim \sup_{Q\in \mathcal{D}^V} |Q|^{\frac{\alpha p'}{n}}$.

It remains to show the lower bound.
For any fixed cube $Q_0\in\mathcal{D}^V$, take the sparse family $\mathcal{S}^V=\{Q_0\}$. Then
\begin{equation*}
C_{p',\alpha}^V(dx)
\geq 
\frac{1}{|Q_0|} \sum_{\substack{Q'\subseteq Q_0 \\ Q'\in \mathcal{S}^V}}|Q'|^{\frac{\alpha p'}{n} + 1} 
\geq
\frac{1}{|Q_0|} |Q_0|^{\frac{\alpha p'}{n} + 1}
=
|Q_0|^{\frac{\alpha p'}{n}}.
\end{equation*}
Taking the supremum over all $Q_0\in\mathcal{D}^V$ gives
$C_{p',\alpha}^V(dx) \gtrsim \sup_{Q\in \mathcal{D}^V} |Q|^{\frac{\alpha p'}{n}}$,
which finishes the proof of (i).

We next prove (ii).
Since $p>n/\alpha$, we have $\alpha p'/n-(p'-1)>0$.
For any $Q\in\mathcal{S}^V$, applying Lemma \ref{l4.101} again 
with $\delta=(\alpha p'/n)-(p'-1)$ and $\gamma=p'$ gives
\begin{equation*}
\frac{1}{\mu(Q)}\sum_{\substack{Q'\subseteq Q \\ Q'\in \mathcal{S}^V}}
|Q'|^{\frac{\alpha p'}{n}-(p'-1)} \mu(Q')^{p'}
\lesssim
|Q|^{\frac{\alpha p'}{n}-(p'-1)} \mu(Q)^{p'-1}.
\end{equation*}
Taking the supremum over all $Q\in\mathcal{S}^V$ yields the upper bound
$$C_{p',\alpha}^V(\mu) \lesssim 
\sup_{Q\in \mathcal{D}^V} |Q|^{\frac{\alpha p'}{n}-(p'-1)} \mu(Q)^{p'-1}.$$
The lower bound follows from an argument similar to the proof of (i), the details being omitted.
This proves (ii) and hence finishes the proof of Lemma \ref{pp4.002}.
\end{proof}

We now present the boundedness estimate for the sparse Riesz potential.

\begin{proposition}\label{l4.10}
Let $p'\in (1,\infty)$, $\alpha\in (0,n)$ and $V\in RH^{n/2}$ satisfy Assumptions 
$\mathbf{(A_1)}$--$\mathbf{(A_3)}$. Suppose $\mu$ is a nonnegative locally finite 
Borel measure on $\rn$ and $C_{p',\alpha}^V(\mu)$ the Carleson constant defined as in \eqref{d4.002}. 
Then for any $g\in C_c^\infty(\rn)$, it holds
\begin{equation*}
\left\| I_{\alpha}^{\mathcal{S}^V} g  \right\|_{L^{p'}(dx)}\lesssim
\left[C_{p',\alpha}^V(\mu)\right]^\frac{1}{p'} \left\|  g  \right\|_{L^{p'}(d\mu)},
\end{equation*}
where the implicit constant is independent of $g$.
\end{proposition}

\begin{proof}   
We first make the following claim:
let $\widetilde{\mathcal{D}}^V$ be an arbitrary subset of the 
perturbed dyadic cube system $\mathcal{D}^V$. For any $s\geq1$, 
there exists a constant $C(s)>0$, depending only on $s$, such that for any 
family of positive real numbers $\{\lambda_Q\}_{Q\in\widetilde{\mathcal{D}}^V}$, it holds
\begin{equation}\label{eq:claim}
A_1:=\int_{\mathbb{R}^n}\biggl(
\sum_{Q\in\widetilde{\mathcal{D}}^V}\frac{\lambda_Q}{|Q|}\mathbf{1}_{N(Q)}(x)
\biggr)^s dx
\leq C(s)\sum_{Q\in\widetilde{\mathcal{D}}^V}\lambda_Q
\biggl(
\frac{1}{|Q|}\sum_{Q'\subset Q}\lambda_{Q'}
\biggr)^{s-1}=:C(s)A_2,
\end{equation}
where $N(Q)$ is defined as in Definition~\ref{d4.50}.

Assuming for the moment that the above claim holds.
By Definition~\ref{d4.7} and substitute $s = p'$, $\widetilde{\mathcal{D}}^V=
\mathcal{S}^V$, 
$\lambda_Q = |Q|^{\frac{\alpha}{n}} \int_Q |g(y)| \, d\mu(y)$ 
into \eqref{eq:claim}, 
we have
\begin{align}\label{e1.1-l4.10}
\int_{\mathbb{R}^n} \big|I_{\alpha}^{\mathcal{S}^V} g(x)\big|^{p'} dx
&= 
\int_{\mathbb{R}^n} \biggl(
\sum_{Q\in \mathcal{S}^V} 
|Q|^{\frac{\alpha}{n}-1} \int_{Q} |g(y)| \, d\mu(y) \mathbf{1}_{N(Q)}(x)
\biggr)^{p'} dx \\ \notag
&\lesssim
\sum_{Q\in\mathcal{S}^V}|Q|^{\frac{\alpha}{n}}\int_Q|g(y)|\,d\mu(y)
\left(
\frac{1}{|Q|}\sum_{Q'\subseteq Q}|Q'|^{\frac{\alpha}{n}}\int_{Q'}|g(y)|\,d\mu(y)
\right)^{p'-1}.
\end{align}
Combining \eqref{e1.1-l4.10} with Lemma \ref{lem3.x6}, we derive
\begin{align}\label{e1.3-l4.10}
\int_{\mathbb{R}^n} \big|I_{\alpha}^{\mathcal{S}^V} g(x)\big|^{p'} dx
&\lesssim 
\sum_{Q\in \mathcal{S}^V} |Q|^{\frac{\alpha}{n}} \int_{Q} |g(y)| \, d\mu(y) 
\left(\frac{1}{|Q|}  |Q|^{\frac{\alpha}{n}} \int_{Q} |g(y)| \, d\mu(y) \right)^{p' -1} \\ \notag
&\sim
\sum_{Q\in \mathcal{S}^V} |Q|^{\frac{\alpha p'}{n}-(p'-1)}\mu(Q)^{p'}
\left(\frac{1}{\mu(Q)} \int_{Q} |g(x)| \, d\mu(x)\right)^{p'}.
\end{align}
We now let $a_Q := |Q|^{\frac{\alpha p'}{n}-(p'-1)}\mu(Q)^{p'}$. 
In light of \eqref{e1.3-l4.10}, to finish the proof of Proposition \ref{l4.10}, 
it suffices to prove the weighted inequality
\begin{align}\label{e1-l4.10}
\sum_{Q\in \mathcal{S}^V} a_Q
\left(\frac{1}{\mu(Q)} \int_{Q} |g(x)| \, d\mu(x)\right)^{p'}
\lesssim
C_{p',\alpha}(V, \mu) \int_{\mathbb{R}^n} |g(x)|^{p'} \,d\mu(x).
\end{align}
To this end, for any  $Q\in \mathcal{S}^V$, let
\begin{equation*}
T_Q g:= \frac{1}{\mu(Q)}\int_{Q} {|g(x)|d\mu(x)}.
\end{equation*}
For any $t>0$, 
let $\{ Q_k^V \}_k$ be the maximal dyadic cubes from the collection 
$\mathcal{S}_t^V(g):=\{Q\in \mathcal{S}^V:\ T_Q g>t\}$.
If $C_{p',\alpha}(V, \mu)<\infty$, then by the definition of $a_Q $ and 
\eqref{d4.002}, we have 
\begin{align*}
\sum\limits_{Q\in \mathcal{S}_t^V(g)} a_{Q}
&= \sum\limits_{k=1}^\infty \sum_{\substack{Q'\in \mathcal{S}_t^V(g)  \\ Q'\subseteq Q_k^V}} a_{Q'}
\lesssim C_{p',\alpha}(V, \mu) \sum\limits_{k=1}^\infty \mu(Q_k^V) \\
&\lesssim \frac{C_{p',\alpha}(V, \mu)}{t} \sum\limits_{k=1}^\infty \int_{{Q_k^V}} {|g(x)| d\mu(x)}
\lesssim \frac{C_{p',\alpha}(V, \mu)}{t} \int_{\rn} {|g(x)| d\mu(x)},
\end{align*}
which implies that $T_Q$ is bounded from to $L^1 (\mu)$ to weak $l^1 (a_{Q})$. 
Furthermore, it is easy to check that
$T_Q$ is bounded from $L^\infty(\mu)$ to $l^\infty(a_Q)$. Therefore, by interpolation, 
\eqref{e1-l4.10} holds.

It remains to prove the claim \eqref{eq:claim}.
The case $s=1$ is trivial. We consider two cases based on the size of $s$.

{\bf Case 1}: $1<s\le2$. In this case, by Lemma \ref{l4.83}, we have
\begin{align*}
A_1
=\int_{\mathbb{R}^n}\left(
\sum_{Q\in\widetilde{\mathcal{D}}^V}\frac{\lambda_Q}{|Q|}\mathbf{1}_{N(Q)}(x)
\right)^s dx
\lesssim\sum_{Q\in\widetilde{\mathcal{D}}^V}\frac{\lambda_Q}{|Q|}
\int_Q\left(
\sum_{Q'\subset Q}\frac{\lambda_{Q'}}{|Q'|}\mathbf{1}_{N(Q')}(x)
\right)^{s-1}dx.
\end{align*}
Then apply H\"older's inequality and the volume equivalence $|Q'| \sim |N(Q')|$,  we derive
\begin{equation*}
\frac{1}{|Q|}\int_{Q}\left(\sum_{Q'\subset Q}\frac{\lambda_{Q'}}{|Q'|}\mathbf{1}_{N(Q')}(x)\right)^{s-1}dx 
\lesssim \left(\frac{1}{|Q|}\int_{Q}\sum_{Q'\subset Q}\frac{\lambda_{Q'}}{|Q'|}\mathbf{1}_{N(Q')}(x)dx\right)^{s-1}
\sim \left(\frac{1}{|Q|}\sum_{Q'\subset Q}\lambda_{Q'}\right)^{s-1}.
\end{equation*}
Combining these two estimates, we find
\begin{equation*}
A_1
\lesssim\sum_{Q\in\widetilde{\mathcal{D}}^V}\lambda_Q
\biggl(
\frac{1}{|Q|}\sum_{Q'\subset Q}\lambda_{Q'}
\biggr)^{s-1}\sim A_2,
\end{equation*}
which completes the proof of the claim \eqref{eq:claim} for $1<s\leq 2$.

{\bf Case 2}: $s>2$. In this case, we proceed by induction. 
Let $k\ge2$ be an integer, and assume the claim \eqref{eq:claim} holds for all $s\in(k-1,k]$.
We verify it for $s\in(k,k+1]$. By Lemma~\ref{l4.83} again, we obtain
\begin{align*}
A_1 &= \int_{\rn}\left(\sum_{Q\in \widetilde{\mathcal{D}}^{V}}\frac{\lambda_{Q}}{|Q|}\mathbf{1}_{N(Q')}(x)\right)^{s}dx
\lesssim \sum_{Q\in \widetilde{\mathcal{D}}^{V}}\frac{\lambda_{Q}}{|Q|}\int_{Q}\left(\sum_{Q'\subset Q}\frac{\lambda_{Q'}}{|Q'|}\mathbf{1}_{N(Q')}(x)\right)^{s-1}dx.
\end{align*}
For any $Q\in \widetilde{\mathcal{D}}^{V}$, 
let $ \widetilde{\mathcal{D}}_Q^{V}:=\{Q'\in \widetilde{\mathcal{D}}^{V}:\ Q'\subset Q\}$.
Applying the induction hypothesis for the claim \eqref{eq:claim}
for $k-1<s-1\leq k$, we have
\begin{align}\label{eqn4.32}
&\sum_{Q\in\widetilde{\mathcal{D}}^{V}}\frac{\lambda_{Q}}{|Q|}\int_{Q}\left(\sum_{Q^{\prime}\subset Q}
\frac{\lambda_{Q^{\prime}}}{|Q^{\prime}|}\mathbf{1}_{N(Q')}(x)\right)^{s-1}dx\\ \notag
&\quad\lesssim \sum_{Q\in \widetilde{\mathcal{D}}^{V}}\frac{\lambda_{Q}}{|Q|}\sum_{Q^{\prime}\in 
\widetilde{\mathcal{D}}_Q^{V}}\lambda_{Q^{\prime}}\left(\frac{1}{|Q^{\prime}|}\sum_{Q^{\prime\prime}\subset 
Q^{\prime}}\lambda_{Q^{\prime\prime}}\right)^{s-2}\\ \notag
&\quad\sim \sum_{Q^{\prime}\in \widetilde{\mathcal{D}}^{V}}\lambda_{Q^{\prime}}\left(\frac{1}{|Q^{\prime}|}
\sum_{Q^{\prime\prime}\subset Q^{\prime}}\lambda_{Q^{\prime\prime}}\right)^{s-2}\sum_{Q^{\prime}\subset Q}\frac{\lambda_{Q}}{|Q|}\\ \notag
&\quad\lesssim \int_{\rn}\sum_{Q^{\prime}\in \widetilde{\mathcal{D}}^{V}}\frac{\lambda_{Q^{\prime}}}{|Q^{\prime}|}\mathbf{1}_{N(Q')}(x)
\left(\frac{1}{|Q^{\prime}|}\sum_{Q^{\prime\prime}\subset Q^{\prime}}\lambda_{Q^{\prime\prime}}\right)^{s-2}
\left(\sum_{Q\in \widetilde{\mathcal{D}}^{V}}\frac{\lambda_{Q}}{|Q|}\mathbf{1}_{N(Q)}(x)\right)dx.
\end{align}
By H\"older's inequality with exponents $s-1$ and $(s-1)/(s-2)$ (note that $s-1>k-1\geq 1$), we have
\begin{align*}
&\sum_{Q^{\prime}\in \widetilde{\mathcal{D}}^{V}}\frac{\lambda_{Q^{\prime}}}{|Q^{\prime}|}\mathbf{1}_{N(Q')}(x)
\left(\frac{1}{|Q^{\prime}|}\sum_{Q^{\prime\prime}\subset Q^{\prime}}\lambda_{Q^{\prime\prime}}\right)^{s-2} \\
&\quad\leq \left(\sum_{Q^{\prime}\in \widetilde{\mathcal{D}}^{V}} \frac{\lambda_{Q^{\prime}}}{|Q^{\prime}|}\mathbf{1}_{N(Q')}(x)\right)^{1/(s-1)}
\left(\sum_{Q^{\prime}\in \widetilde{\mathcal{D}}^{V}}\frac{\lambda_{Q^{\prime}}}{|Q^{\prime}|}
\mathbf{1}_{N(Q')}(x)\left(\frac{1}{|Q^{\prime}|}\sum_{Q^{\prime\prime}\subset Q^{\prime}}\lambda_{Q^{\prime\prime}}\right)^{s-1}\right)^{(s-2)/(s-1)}.
\end{align*}
Substituting this estimate into the right-hand side of \eqref{eqn4.32}, we obtain
\begin{align*}
A_1 \lesssim&
\int_{\rn}\left(\sum_{Q^{\prime}\in\widetilde{\mathcal{D}}^{V}}\frac{\lambda_{Q^{\prime}}}{|Q^{\prime}|}
\mathbf{1}_{N(Q)}(x)\right)^{1/(s-1)+1}\left(\sum_{Q^{\prime}\in \widetilde{\mathcal{D}}^{V}}\frac{\lambda_{Q^{\prime}}}{|Q^{\prime}|}
\mathbf{1}_{N(Q)}(x)\left(\frac{1}{|Q^{\prime}|}\sum_{Q^{\prime\prime}\subset Q^{\prime}}\lambda_{Q^{\prime\prime}}\right)^{s-1}\right)^{(s-2)/(s-1)}dx.
\end{align*}
Applying now H\"older's inequality again with exponents $s-1$ and $(s-1)/(s-2)$, we conclue
\begin{align*}
A_1&\lesssim \left(\int_{\rn}\left(\sum_{Q^{\prime}}\frac{\lambda_{Q^{\prime}}}{|Q^{\prime}|}
\mathbf{1}_{N(Q')}(x)\right)^{s}dx\right)^{1/(s-1)}\\
&\quad\times\left(\int_{\rn}\sum_{Q^{\prime}}\frac{\lambda_{Q^{\prime}}}{|Q^{\prime}|}
\mathbf{1}_{N(Q')}(x)\left(\frac{1}{|Q^{\prime}|}\sum_{Q^{\prime\prime}\subset Q^{\prime}}\lambda_{Q^{\prime\prime}}\right)^{s-1}dx\right)^{(s-2)/(s-1)}
\sim A_1^{1/(s-1)}A_2^{(s-2)/(s-1)},
\end{align*}
which gives $A_1\lesssim A_2$ and hence completes the proof of the claim \eqref{eq:claim}.
This finishes the proof of Proposition \ref{l4.10}.
\end{proof}

The following proposition provides an upper bound for the operator norm of the 
Riesz potential $(-\Delta+V)^{-{\alpha}/{2}}$ in terms of the Carleson constant.

\begin{proposition}\label{p4.003}
Let $\alpha\in (0,n)$, $1<p<\infty$ and $V\in RH^{n/2}$ satisfy Assumptions $\mathbf{(A_1)}$--$\mathbf{(A_3)}$.
Suppose $\mu$ is a nonnegative locally finite Borel measure on $\mathbb{R}^n$. 
Then it holds
\begin{equation*}
\left\|(-\Delta+V)^{-\frac{\alpha}{2}}\right\|_{L^p(dx)\to L^p(d\mu)}^p
\lesssim\! C_{p',\alpha}^V(\mu).
\end{equation*}
\end{proposition}

\begin{proof}
By Proposition \ref{l4.4}, Lemma \ref{l4.6}, and Propositions \ref{l4.8}, \ref{l4.10}, 
we obtain the chain of inequalities
\begin{align*}
\biggl\|(-\Delta+V)^{-\frac{\alpha}{2}}(g\mu) \biggr\|_{L^{p'}(dx)} 
&\lesssim
\biggl\|  I_{\alpha,\mathrm{loc}}^V  g \biggr\|_{L^{p'}(dx)} 
\lesssim
\biggl\| I_{\alpha}^{\mathcal{D}^V} g  \biggr\|_{L^{p'}(dx)} \\
&\lesssim
\biggl\| I_{\alpha}^{\mathcal{S}^V} g  \biggr\|_{L^{p'}(dx)} 
\lesssim
\bigl[C_{p',\alpha}^V(\mu)\bigr]^\frac{1}{p'} \|  g \|_{L^{p'}(d\mu)}.
\end{align*}
Combining this estimate with duality, we conclude that
\begin{equation*}
\biggl\|(-\Delta+V)^{-\frac{\alpha}{2}} \biggr\|_{L^p(dx)\to L^p(d\mu)}
= \biggl\|(-\Delta+V)^{-\frac{\alpha}{2}}(\cdot\mu) \biggr\|_{L^{p'}(d\mu)\to L^{p'}(dx)}
\lesssim \bigl[C_{p',\alpha}^V(\mu)\bigr]^\frac{1}{p'} ,
\end{equation*}
which completes the proof of Proposition \ref{p4.003}.
\end{proof}

To obtain a lower bound estimate for the operator norm of the Riesz potential via the Carleson constant, 
we need the following two lemmas.

\begin{lemma}\label{l4.21}
Let $0<\alpha\leq 2$ and $V\in RH^{n/2}$. Then the integral 
kernel $K_\alpha^V$ of the Riesz potential 
$(-\Delta+V)^{-\alpha/2}$ satisfies the following lower bound: for any $x,y\in\mathbb{R}^n$,
it holds
\begin{align}\label{eqn-UB}
K_\alpha^V(x,y)
\gtrsim
\frac{e^{-\varepsilon \bigl(1+|x-y|m(x,V)\bigr)^{k_0+1}}}{|x-y|^{n-\alpha}},
\end{align}
where $m(x,V)$ is defined as in \eqref{eqn-rho} and 
$\varepsilon>0$, $k_0>0$ are constants independent of $x$ and $y$.
\end{lemma}

\begin{proof} 
The estimate \eqref{eqn-UB} is obviously satisfied for $\alpha=2$ owing to 
the pointwise bound on the fundamental solution in Theorem {\bf C}.
We now consider the case $\alpha\in(0,2)$.
The fractional integral operators are well-known to 
have the following representation (see, e.g., \cite[p.~286]{k1995}),
\begin{align*}
(-\Delta+V)^{-\frac{\alpha}{2}}f(x) &= \frac{1}{\pi}\sin\left(\frac{\pi\alpha}{2}\right)\int_0^\infty \lambda^{-\alpha/2}(-\Delta+\mu+\lambda)^{-1}f(x)\,d\lambda \\
&= \frac{1}{\pi}\sin\left(\frac{\pi\alpha}{2}\right)\int_0^\infty \lambda^{-\alpha/2}\int_{\mathbb{R}^d}\Gamma_{\mu+\lambda}(x,y)f(y)\,dy\,d\lambda \\
&= \int_{\mathbb{R}^d}\left(\frac{1}{\pi}\sin\left(\frac{\pi\alpha}{2}\right)\int_0^\infty \lambda^{-\alpha/2}\Gamma_{\mu+\lambda}(x,y)\,d\lambda\right)f(y)\,dy .
\end{align*}
where $\Gamma_{V+\lambda}(x,y)$ stands for the 
fundamental solution of $-\Delta+V+\lambda$ on $\mathbb{R}^n$.
Thus, the kernel $K_\alpha^V(x,y)$ satisfies
\begin{equation}\label{e2.1-l4.21}
K_\alpha^V(x,y)
\sim \int_0^\infty \lambda^{-\frac{\alpha}{2}} \Gamma_{V+\lambda}(x,y) \,d\lambda.
\end{equation}
By Theorem {\bf C}, we have
\begin{equation}\label{e2.2-l4.21}
\Gamma_{V+\lambda}(x,y)\gtrsim 
\frac{e^{-\varepsilon (1+|x-y|m(x,V+\lambda))^{k_{0}+1}}}{|x-y|^{n-2}}.
\end{equation}
Applying Lemma~\ref{l2.9}, we know that 
$m(x,V+\lambda)\sim \max\{m(x,V), m(x,\lambda)\}$, 
which implies
\begin{equation}\label{e2.3-l4.21}
(1+|x-y|m(x,V+\lambda))^{k_{0}+1} \lesssim
(1+|x-y|m(x,V))^{k_{0}+1} + (1+|x-y|m(x,\lambda))^{k_{0}+1}.
\end{equation}
Combining \eqref{e2.1-l4.21}, \eqref{e2.2-l4.21}, \eqref{e2.3-l4.21} and the fact that $m(x,\lambda)\sim \sqrt{\lambda}$, we deduce that
\begin{align*}
K_\alpha^V(x,y)
&\gtrsim \int_0^\infty \lambda^{-\frac{\alpha}{2}} 
\frac{e^{-\varepsilon (1+|x-y|m(x,V))^{k_{0}+1}} 
e^{-\varepsilon (1+|x-y|m(x,\lambda))^{k_{0}+1}}} {|x-y|^{n-2}} \,d\lambda \\
&\sim \frac{e^{-\varepsilon (1+|x-y|m(x,V))^{k_{0}+1}}}{|x-y|^{n-2}} 
\int_0^\infty \lambda^{-\frac{\alpha}{2}} e^{-\varepsilon (1+|x-y|\sqrt{\lambda})^{k_{0}+1}} \,d\lambda \\
&\sim \frac{e^{-\varepsilon (1+|x-y|m(x,V))^{k_{0}+1}}}{|x-y|^{n-2}} \cdot \frac{1}{|x-y|^{2-\alpha}}
\sim \frac{e^{-\varepsilon (1+|x-y|m(x,V))^{k_{0}+1}}}{|x-y|^{n-\alpha}}.
\end{align*}
This completes the proof of \eqref{eqn-UB}.
\end{proof}

The following lemma follows from a classical telescoping argument, the details being 
omitted.
\begin{lemma}\label{l4.210}
Let $\alpha\in(0,n)$ and $V\in RH^{n/2}$ satisfy Assumptions $\mathbf{(A_1)}$--$\mathbf{(A_3)}$.
Suppose $\mu$ is a nonnegative locally finite Borel measure on $\mathbb{R}^n$
and $\mathcal{D}^V$ the perturbed dyadic cube system constructed in Theorem~\ref{t1.2}. 
For any cube $Q\in\mathcal{D}^V$, let $\mathcal{D}^V(Q):=\{Q'\in\mathcal{D}^V:\ Q'\subset Q\}$.
Then, for any $g\in C_c^\infty(\rn)$ and any $x\in Q$, it holds
\begin{equation*}
\int_{Q} \frac{|g(y)|}{|x-y|^{n-\alpha}}\,d\mu(y)
\gtrsim
\sum_{Q'\in\mathcal{D}^V(Q)} |Q'|^{\frac{\alpha}{n}-1} \int_{Q'} |g(y)|\,d\mu(y)\mathbf{1}_{Q'}(x).
\end{equation*}
\end{lemma}

With the help of the above lemmas, we now give 
a lower bound for the operator norm of the Riesz potential in terms of the Carleson constant.

\begin{proposition}\label{p4.0031}
Let $\alpha \in (0,2]$, $1<p<\infty$ and $V\in RH^{n/2}$ 
satisfy Assumptions $\mathbf{(A_1)}$--$\mathbf{(A_3)}$.
Suppose $\mu$ is a nonnegative locally finite Borel measure on $\mathbb{R}^n$. 
Then
\begin{equation*}
\biggl\|(-\Delta+V)^{-\frac{\alpha}{2}}\biggr\|_{L^p(dx)\to L^p(d\mu)}^p
\gtrsim\! C_{p',\alpha}^V(\mu).
\end{equation*}
\end{proposition}

\begin{proof}
First, by  duality, we have
\begin{equation}\label{e2.0-p4.0031}
\biggl\|(-\Delta+V)^{-\frac{\alpha}{2}} \biggr\|_{L^p(dx)\to L^p(d\mu)}
= \biggl\|(-\Delta+V)^{-\frac{\alpha}{2}}(\cdot\mu) \biggr\|_{L^{p'}(d\mu)\to L^{p'}(dx)}.
\end{equation}
Moreover, for any $Q\in \mathcal{D}^V$, since $\mathbf{1}_{Q}/\mu(Q)^{1/p'}\in L^{p'}(\mu)$
with $\|\mathbf{1}_{Q}/\mu(Q)^{1/p'}\|_{L^{p'}(\mu)}=1$.
it is straightforward to verify that
\begin{equation}\label{e2-p4.0031}
\biggl\|(-\Delta+V)^{-\frac{\alpha}{2}}(\cdot\mu) \biggr\|^{p'}_{L^{p'}(d\mu)\to L^{p'}(dx)}
\geq
\sup\biggl\{
\frac{1}{\mu(Q)} \int_Q \biggl|(-\Delta+V)^{-\frac{\alpha}{2}}(\mathbf{1}_Q \mu)(x)\biggr|^{p'} dx:Q\in \mathcal{D}^{V}
\biggr\}.
\end{equation}
Note that for any $x,y\in Q\in\mathcal{D}^V$, it holds $|x-y|<\rho(x,V)$.
This combined with Lemma~\ref{l4.21} implies the pointwise lower bound 
$$K_\alpha^V(x,y)\gtrsim |x-y|^{\alpha-n},$$ from which and 
Lemma~\ref{l4.210}, we further obtain
\begin{align*} 
\frac{1}{\mu(Q)} \int_Q \biggl|(-\Delta+V)^{-\frac{\alpha}{2}}(\mathbf{1}_Q \mu)(x)\biggr|^{p'} 
dx& \gtrsim 
\frac{1}{\mu(Q)} \int_Q \biggl(\int_{Q} \frac{d\mu(y)}{|x-y|^{n-\alpha}} \biggr)^{p'} 
dx\\ 
&\gtrsim\frac{1}{\mu(Q)} \int_{\mathbb{R}^n} \biggl( \sum_{Q'\in \mathcal{D}^{V}} |Q'|^{\frac{\alpha}{n}-1}
\int_{Q'} d\mu(y)\,\mathbf{1}_{Q'} (x)\biggr)^{p'}dx.
\end{align*}

Now, assume that $C_{p',\alpha}^V(\mu)$ is associated with a sparse family 
$\mathcal{S}^V\subseteq \mathcal{D}^{V}$. By the monotonicity of $\ell^r$-norms,
we have
\begin{align*}  
\frac{1}{\mu(Q)} \int_{\mathbb{R}^n} \biggl( \sum_{Q'\in \mathcal{D}^{V}} |Q'|^{\frac{\alpha}{n}-1}
\int_{Q'} d\mu(y)\,\mathbf{1}_{Q'} (x)\biggr)^{p'}dx \gtrsim
\frac{1}{\mu(Q)}
\sum_{\substack{Q'\subseteq Q \\ Q'\in \mathcal{S}^V}}|Q'|^{\frac{\alpha p'}{n}-(p'-1)} \mu(Q')^{p'},
\end{align*}
which combined with \eqref{d4.002} and \eqref{e2-p4.0031} implies
\begin{align}\label{e1.6-p4.0031}
\biggl\|(-\Delta+V)^{-\frac{\alpha}{2}}(\cdot\mu) \biggr\|^{p'}_{L^{p'}(d\mu)\to L^{p'}(dx)}&\gtrsim
C_{p',\alpha}^V(\mu).
\end{align}
By this and \eqref{e2.0-p4.0031},
we complete the proof of Proposition~\ref{p4.0031}.
\end{proof}

We are now ready to prove Theorem \ref{p4.002}.

\begin{proof}[\bf Proof of Theorem \ref{p4.002}]
Theorem \ref{p4.002} follows immediately from Propositions 
\ref{p4.003}, \ref{p4.0031} and Lemma \ref{pp4.002}.
\end{proof}

\subsection{Proof of Corollary \ref{t1.3}}\label{s4.4}

In this subsection, we prove Corollary \ref{t1.3}.
We start with the following representation of 
$\inf \sigma(-\Delta +V)$ in terms of Riesz potential.

\begin{lemma}\label{p4.001}
Let $V\in RH^q$ with $q>1$. Then
\begin{equation*}
\inf \sigma(-\Delta +V)=\left\| (-\Delta +V)^{-\frac{1}{2}} \right\|_{L^2(dx)\rightarrow L^2(dx)}^{-2}. 
\end{equation*}
\end{lemma}

\begin{proof}   
For simplicity, we write the $L^2$-norm $\| \cdot \|_{L^2(dx)}$ as $\| \cdot \|_2$.
Let $L:= -\Delta +V$ and $\lambda_0 := \inf \sigma(H)$. 
By Dirichlet principle (see, e.g., \cite[Corollary 1.11]{flw2023}), we have
\begin{equation}\label{e1-p4.001}
\lambda_0 
= \inf\limits_{u\neq 0} \frac{\langle Lu,u \rangle}{\| u \|_2} 
= \inf\limits_{u\neq 0} \frac{\| L^\frac{1}{2}u \|_2^2}{\| u \|_2^2}. 
\end{equation}
Let $f:= L^{1/2}u$. Then by \eqref{e1-p4.001}, we obtain
\begin{equation*}
\lambda_0 
= \inf\limits_{f\neq 0} \frac{\| f \|_2^2}{\| L^{-\frac{1}{2}}f \|_2^2} 
=\frac{1}{\sup\limits_{f\neq 0} \frac{{\| L^{-\frac{1}{2}}f \|_2^2}}{{\| f \|_2^2}} }
=\left\| L^{-\frac{1}{2}} \right\|_{L^2(dx)\rightarrow L^2(dx)}^{-2},
\end{equation*}
which completes the proof of Lemma \ref{p4.001}.
\end{proof}

We are now ready to prove Corollary~\ref{t1.3}.

\begin{proof}[\bf{Proof of Corollary \ref{t1.3}}]
For simplicity, let $\lambda_1(L):=\inf \sigma (L)$. 
We first prove part (i).
By Lemma~\ref{p4.001}, we have  
\begin{equation*}
\lambda_1(L) = \left\|L^{-1/2}\right\|_{L^2(\mathbb{R}^n;dx)\to L^2(\mathbb{R}^n;dx)}^{-2}.
\end{equation*}
Now applying Theorem \ref{p4.002} with $p=2$ and $\alpha=1$, we know 
\begin{equation*}
\lambda_1(L)
= \left\|L^{-1/2}\right\|_{L^2\to L^2}^{-2}
\sim \left(\sup_{Q\in\mathcal{D}^V} |Q|^{2/n}\right)^{-1}
\sim \inf_{Q\in\mathcal{D}^V} |Q|^{-2/n},
\end{equation*}
which verifies (i).

We next prove (ii).
For any $\lambda>\lambda_1(L)$, since $V+\lambda \in RH^{n/2}$ also satisfies 
Assumptions $\mathbf{(A_1)}$–$\mathbf{(A_3)}$, there exists a
system of perturbed dyadic cubes $\mathcal{D}^{V+\lambda}$ as in Definition~\ref{d3.1}.
Moreover, using Lemma~\ref{l2.9} and the fact $\rho(x,\lambda)\sim 
\lambda^{-1/2}$, we have 
\begin{equation*}
\rho(x,V+\lambda) \sim \min\bigl\{\rho(x,V),\,\lambda^{-1/2}\bigr\}.
\end{equation*}
Now, let $\mathcal{D}_0^{V+\lambda}$ be the class of all $0$-th generation dyadic cubes
in $\mathcal{D}^{V+\lambda}$.
We divide cubes in $\mathcal{D}_0^{V+\lambda}$ into the following two kinds
(see Figure 5 for an illustration).
\begin{itemize}
\item [{\rm (i)}] Cubes in \emph{spectral-dominated region}: following the construction in 
Subsection \ref{s3.1},
in this region, all lattice points $\{x_l\}_l$ defined as in \eqref{eqn-LP} satisfy 
$\rho(x_l,V+\lambda)\sim \min\bigl\{\rho(x_l,V),\,\lambda^{-1/2}\bigr\}
\sim \lambda^{-1/2}$. Thus, there exists a $k\in\mathbb{Z}$ such that 
all the cubes have uniform side length $$l(Q)=2^{k-[\log _2\lambda/2]}.$$
\item [{\rm (ii)}] Cubes in \emph{potential-dominated region}: 
similar to (i), all lattice points $\{x_l\}_l$ defined as in \eqref{eqn-LP} satisfy 
$\rho(x_l,V+\lambda)\sim \min\bigl\{\rho(x_l,V),\,\lambda^{-1/2}\bigr\}
\sim \rho(x_l,V)$. Thus, all the cubes in this region satisfy $l(Q)\sim\rho(x_Q,V)$.
\end{itemize}
\begin{figure}[H] %
  \centering %
  \captionsetup{labelformat=empty}  
  \includegraphics[scale=0.1]{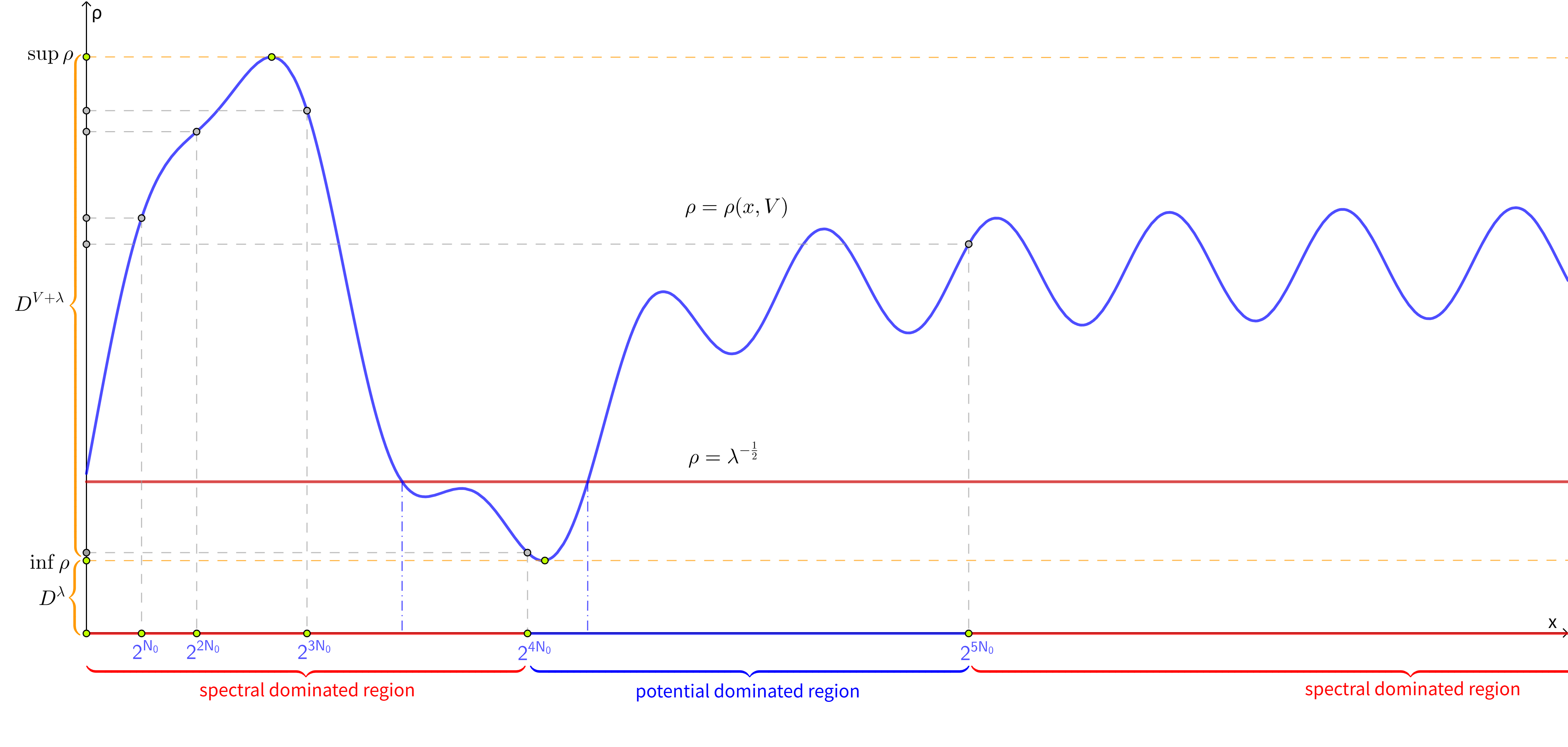} %
  \caption{Figure 5: Spectral dominated region and potential dominated region} %
\end{figure}
On the other hand, by Remark \ref{rem1.5x},  we have 
\begin{equation*}
N(\lambda,L) \sim \lambda^{\frac{n}{2}} \left|\left\{x\in\mathbb{R}^n : \rho(x,V)>\lambda^{-1/2}\right\}
\right|\sim  \lambda^{\frac{n}{2}} \left|\left\{x\in\mathbb{R}^n : \rho(x,V+\lambda)
\sim \lambda^{-1/2}\right\}\right|.
\end{equation*}
Using the facts $|\{x\in\mathbb{R}^n : \rho(x,V+\lambda)
\sim \lambda^{-1/2}\}|$ is comparable to the measure of 
the spectral-dominated region of $\mathcal{D}_0^{V+\lambda}$
and each dyadic cube in this region is of the side length $\sim \lambda^{-1/2}$,
we immediately conclude 
\begin{equation*}
N(\lambda,L) \sim \#\Bigl\{Q\in\mathcal{D}_0^{V+\lambda} : l(Q)=2^{k-[\log_2 \sqrt{\lambda}]}\Bigr\},
\end{equation*}
which proves (ii) and hence completes the proof of Corollary \ref{t1.3}.
\end{proof}

\bigskip

\noindent Jun Cao, Cheng Chen, Chaohong Deng  and Yulian Wu

\medskip

\noindent School of Mathematical Sciences, Zhejiang University of Technology, 
310023, Hangzhou, China

\smallskip

\noindent{\it E-mail address}: 
\texttt{caojun1860@zjut.edu.cn}\quad(Jun Cao)

\hspace{2cm}\texttt{202203160103@zjut.edu.cn}\quad (Cheng Chen)

\hspace{2cm}\texttt{1181104006@qq.com}\quad (Chaohong Deng)

\hspace{2cm}\texttt{2021212047@nwnu.edu.cn}\quad (Yulian Wu)
\end{document}